\documentclass{article}
\usepackage{amsthm}
\usepackage{graphicx,geometry,amsmath,amssymb,stmaryrd,enumerate,xcolor,latexsym}
\usepackage[english]{babel}
\usepackage{url}
\usepackage[normalem]{ulem}
 \usepackage{lineno}
\usepackage{tikz}
\usetikzlibrary{patterns}
\usepackage{pgfplots}
\pgfplotsset{compat=newest}
\usetikzlibrary{calc}

\DeclareMathOperator{\Div}{div}
\newcommand{\Th}{\mathcal{T}_h}
\newcommand{\assign}{:=}
\newcommand{\nobracket}{}

\newcommand{\tmop}[1]{\text{#1}}
\newcommand{\tmtextit}[1]{{\itshape{#1}}}

\newtheorem{lemma}{Lemma}
\newtheorem{theorem}{Theorem}
\newtheorem{rmk}{Remark}
\newtheorem{assumption}{Assumption}

\newcommand{\MD}[1]{\textcolor{black}{#1}}
\newcommand{\vanessa}[1]{\textcolor{black}{#1}}
\newcommand{\AL}[1]{\textcolor{black}{#1}}
\newcommand{\logLogSlopeTriangle}[5]
{

    \pgfplotsextra
    {
        \pgfkeysgetvalue{/pgfplots/xmin}{\xmin}
        \pgfkeysgetvalue{/pgfplots/xmax}{\xmax}
        \pgfkeysgetvalue{/pgfplots/ymin}{\ymin}
        \pgfkeysgetvalue{/pgfplots/ymax}{\ymax}

        \pgfmathsetmacro{\xArel}{#1}
        \pgfmathsetmacro{\yArel}{#3}
        \pgfmathsetmacro{\xBrel}{#1-#2}
        \pgfmathsetmacro{\yBrel}{\yArel}
        \pgfmathsetmacro{\xCrel}{\xArel}

        \pgfmathsetmacro{\lnxB}{\xmin*(1-(#1-#2))+\xmax*(#1-#2)} 
        \pgfmathsetmacro{\lnxA}{\xmin*(1-#1)+\xmax*#1} 
        \pgfmathsetmacro{\lnyA}{\ymin*(1-#3)+\ymax*#3} 
        \pgfmathsetmacro{\lnyC}{\lnyA+#4*(\lnxA-\lnxB)}
        \pgfmathsetmacro{\yCrel}{\lnyC-\ymin)/(\ymax-\ymin)} 
        
        \coordinate (A) at (rel axis cs:\xArel,\yArel);
        \coordinate (B) at (rel axis cs:\xBrel,\yBrel);
        \coordinate (C) at (rel axis cs:\xCrel,\yCrel);

        \draw[#5]   (A)-- node[pos=0.5,anchor=north] {1}
                    (B)-- 
                    (C)-- node[pos=0.5,anchor=west] {#4}
                    cycle;
    }
}

\begin{document}
\title{$\phi$-FEM: an optimally convergent and easily implementable immersed boundary method for particulate flows and Stokes equations\footnote{This work was supported by the Agence Nationale de la Recherche, Project PhiFEM, under grant ANR-22-CE46-0003-01.}}
\author{Michel Duprez\footnote{MIMESIS team, Inria Nancy - Grand Est, MLMS team, Universit\'e de Strasbourg, France. \texttt{michel.duprez@inria.fr}},Vanessa Lleras\footnote{IMAG, Univ Montpellier, CNRS, Montpellier, France.  \texttt{vanessa.lleras@umontpellier.fr}}, Alexei Lozinski\footnote{Universit\'e de Franche-Comt\'e, Laboratoire de Math\'ematiques de Besan\c{c}on, UMR CNRS 6623, 16, route de Gray
25000 Besan\c{c}on, France. \texttt{alexei.lozinski@univ-fcomte.fr}}}
\date{\today}
%

%

%


%
%
%
\maketitle

\begin{abstract} 
We present an immersed boundary method to simulate the creeping motion of a \AL{rigid particle} in a fluid described by the Stokes equations discretized thanks to a finite element strategy on unfitted meshes, called $\phi$-FEM, that uses the description of the solid with a level-set function.
One of the advantages of our method is the use of standard finite element spaces and classical integration tools, while maintaining the optimal convergence (theoretically in the $H^1$ norm for the velocity and $L^2$ for pressure; numerically also in the $L^2$ norm for the velocity).
\end{abstract}

\section{Introduction}

The main goal of the present article is to demonstrate that the recently proposed $\phi$-FEM methodology \cite{phifem,phifem2,cotin2021phi} is suitable for numerical simulation of incompressible viscous fluid flow past moving rigid bodies.  This approach allows us to use simple (Cartesian) computational meshes, not evolving in time and not fitted to the moving rigid bodies,  while achieving the optimal accuracy with classical finite element (FE) spaces of any order and performing the usual numerical integration on the whole mesh cells and facets, allowing for the use of standard FEM libraries for the implementation. We consider here only the creeping motion regime (zero Reynolds number), neglecting all the inertial terms in the equation governing both the fluid and the rigid bodies.

Numerical simulations of flows around moving rigid or elastic structures using immobile simple grids is a popular approach in, for instance, biomechanics, starting from the work of Peskin \cite{peskin77}. Different approaches have emerged since then, such as the Immersed Boundary method \cite{lai00immersed,mittal05immersed}, the Fictitious Domain method  \cite{glowinski99,glowinski01}, the penalty approximation \cite{angot99}, etc. All these classical methods suffer from poor accuracy due to the necessity to approximate the singularities near the fluid-solid interfaces which arise as the artifact of extending the fluid velocity field inside the solid domain. More recently, several optimally convergent fictitious domain-type methods have been proposed  for the Stokes equations, which can also be used to simulate the fluid-solid motions. We cite in particular \cite{burman,massing,guzman} following the CutFEM  paradigm, and \cite{amdouni12,alexei} following the X-FEM paradigm. The common feature of all these methods is that they discretize the variational formulation of the Stokes equation on the physical fluid domain $\Omega$ using the FE spaces defined on the background mesh occupying a domain $\Omega_h$, slightly larger than $\Omega$. On the one hand, this permits to avoid a non-smooth extension of the solution outside its natural domain and to retrieve the optimal accuracy of the employed finite elements. On the other hand, this introduces integrals on the cut cells into the FE scheme, i.e. the numerical integration should be performed on the portions of mesh cells, cut by the fluid-solid interface, making the methods difficult to implement.

The $\phi$-FEM approach, which is the subject of the present paper, aims at combining the advantages of both classical Immersed Boundary/Fictitious Domain methods, and more recent CutFEM/X-FEM. Similarly to the former,  $\phi$-FEM does not need non-standard numerical integration on the cut cells; similarly to the latter, $\phi$-FEM achieves the optimal accuracy of the finite elements employed.  The general procedure of $\phi$-FEM can be summarized as follows:
\begin{itemize}
    \item Supposing that the physical domain $\Omega$ is given by a level set function $\Omega=\{\phi<0\}$ and that it is embedded into a simple background mesh, introduce the \textit{active} computational mesh $\Th$ by getting rid of the mesh cells lying completely outside $\Omega$. The active mesh thus occupies a domain $\Omega_h\supset\Omega$, cf. Fig.~\ref{fig:omega_h},  as in CutFEM/X-FEM.
    \item Extend the governing equations from $\Omega$ to $\Omega_h$ and write down a formal variational formulation on $\Omega_h$ without taking into account the boundary conditions on $\Gamma$ (the relevant part of the boundary of $\Omega$).
    \item Impose the boundary conditions on $\Gamma$ using an appropriate ansatz or some additional variables, explicitly involving the level set $\phi$ which provides the link to the actual boundary. For instance, the homogeneous Dirichlet boundary conditions ($u=0$ on $\Gamma$) can be imposed by the ansatz ${u}=\phi{w}$ thus reformulating the problem in terms of the new unknown ${w}$.
    \item Add appropriate stabilization, typically combining the ghost penalty \cite{ghost} with a least square imposition of the governing equation on the mesh cells intersected by $\Gamma$, to guarantee coerciveness/stability on the discrete level.
\end{itemize}
This program has been successfully carried out for elliptic scalar PDEs with Dirichlet boundary conditions in \cite{phifem} and for Neumann boundary conditions in \cite{phifem2}. Its feasibility is also demonstrated in \cite{cotin2021phi} for the linear elasticity with mixed boundary conditions including the cases of internal interfaces between different materials or cracks, and for the heat equation. However, the adaptation to the equations governing the fluid flow around the moving particles is not straightforward. In particular, the following challenges are dealt with in the present article:
\begin{itemize}
    \item The discrete inf-sup stability theory should be adapted to the case of a non-standard variational formulation of the Stokes equations posed on $\Omega_h$ rather than on $\Omega$, and lacking the saddle-point structure. We shall show that this is possible by adapting the ghost penalty, which should be taken slightly more complicated than in the case of scalar elliptic equations \cite{phifem}. We shall do it here for Taylor-Hood finite elements of any order, but similar ideas should be also applicable to other classical inf-sup stable FE spaces.
    \item The motion equations for the solid particles involve the forces exerted on them by the surrounding fluid. These are defined through the integrals of some functions of fluid velocity and pressure on the particle boundary. However, the whole point of $\phi$-FEM is to avoid such integrals. Indeed, the particle boundary is not resolved by the mesh, and our goal is to provide a method that necessitates the integration on the whole mesh cells or facets only. The way out of this paradox, pursued in the present paper, lies in providing a weak formulation of the governing equations, extended to $\Omega_h$, that incorporates in an appropriate way the force balance equations, without stating them directly.  This formulation is similar in spirit but different from that in \cite{glowinski01}.
\end{itemize}
We note that the method of this article shares some similarities to the shifted boundary method (SBM) proposed in  \cite{sbm} and analysed in \cite{sbm2} in the case of Stokes equations. In particular, SBM also gives an optimal accurate solution (at least with the lowest order finite elements) without introducing integrals on the cut cells. It is however not evident \vanessa{how one can deal} with the computation of the forces on the particles in the SBM framework.

\begin{figure}
    \centering
    \includegraphics[width=0.3\linewidth]{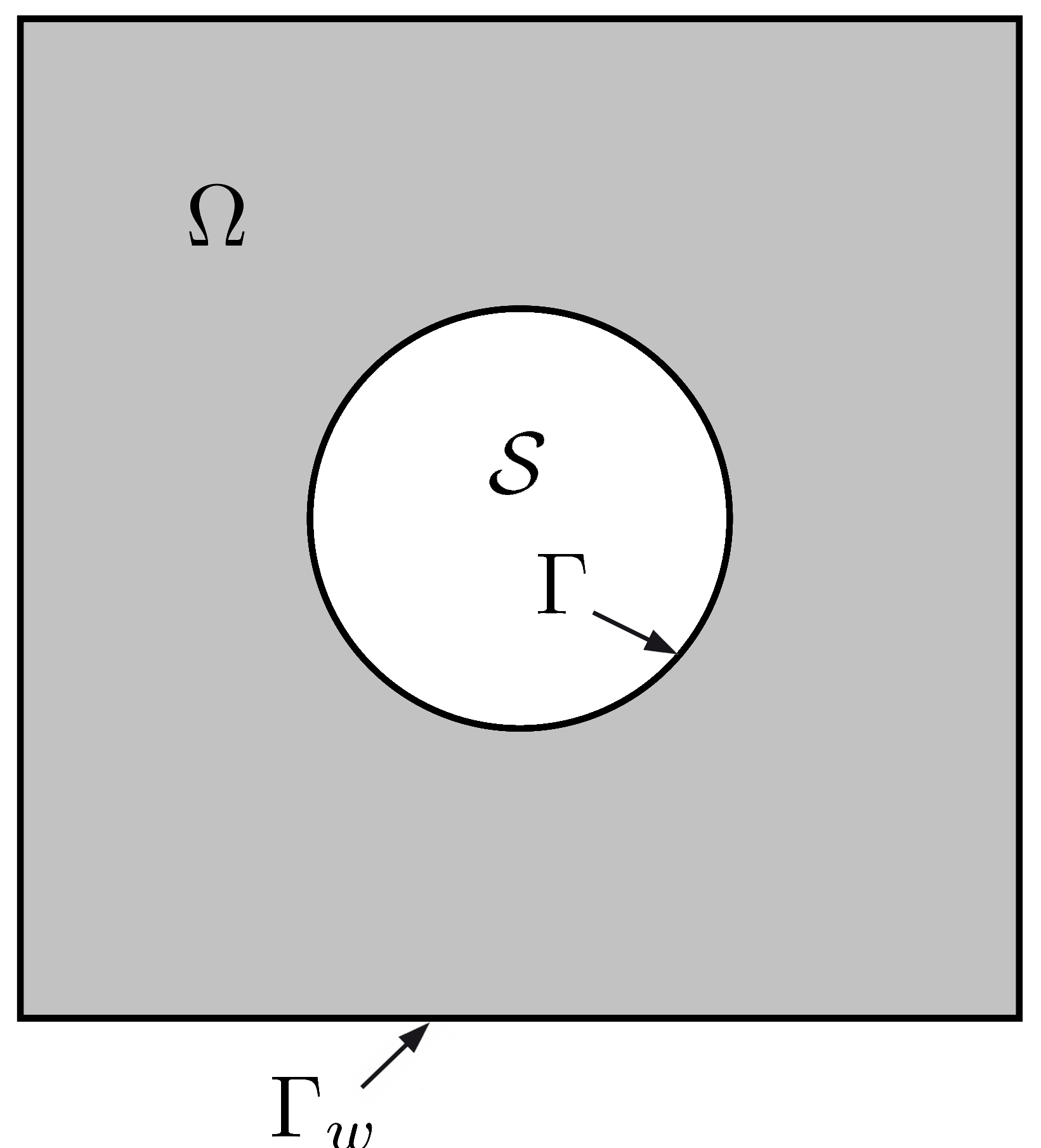}
    \includegraphics[width=0.3\linewidth]{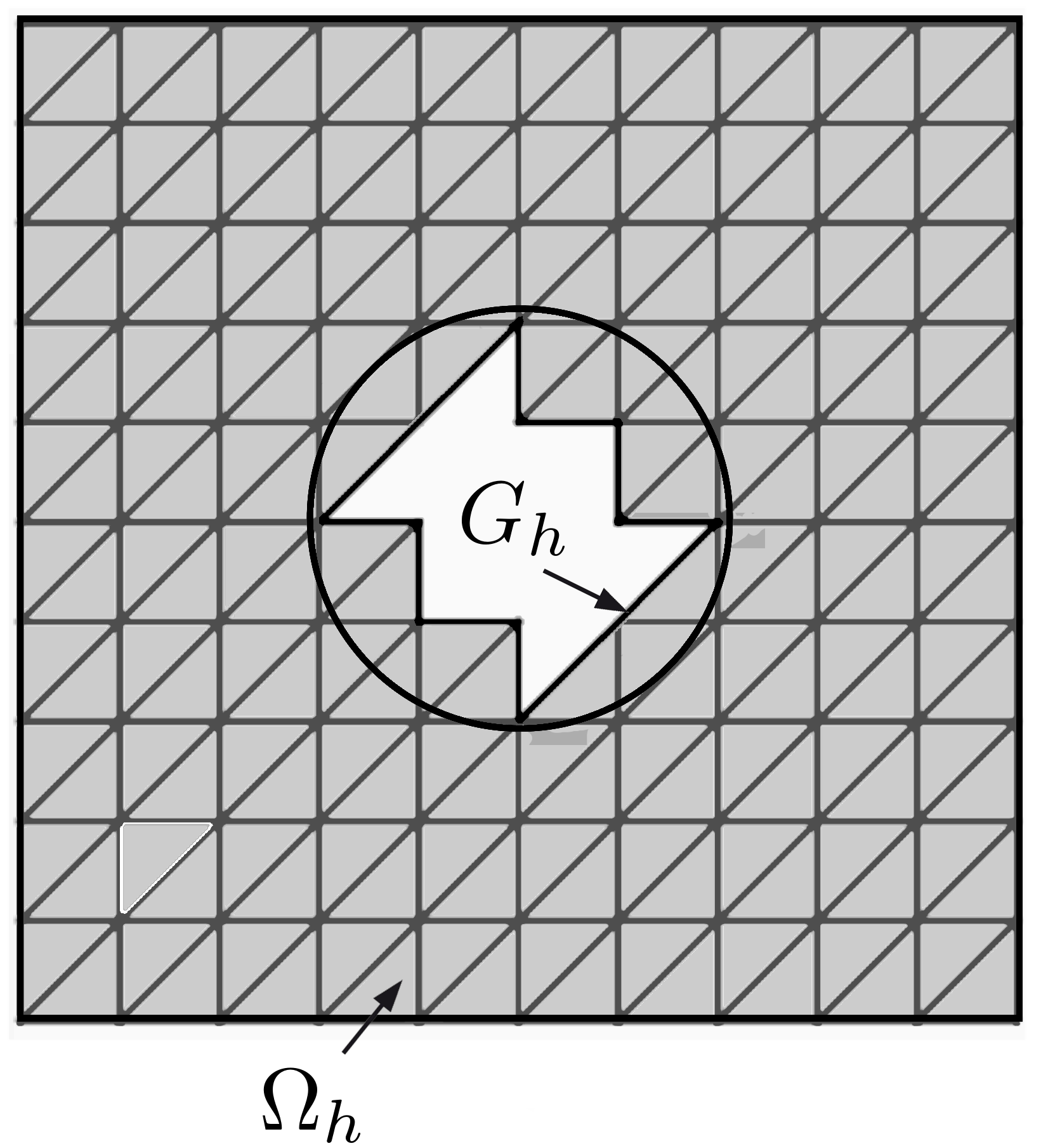}
    \caption{Left:~an example of geometry for the fluid $\Omega$ with a solid $\mathcal{S}$ inside; Right:~the  non-conforming active mesh $\Th$ on $\Omega_h$ with its internal boundary $G_h$.}
    \label{fig:omega_h}
\end{figure}

The paper is structured as follows. In the next section, we properly introduce the governing equations, develop an appropriate weak formulation, discretize it (thus introducing our $\phi$-FEM scheme), and announce the main theorem about the convergence of the scheme. Section 3 is devoted to the proof of this theorem. As a by-product, we also introduce a $\phi$-FEM approach to discretize the Stokes equations alone (on a fixed geometry) on a non-fitted mesh. The details about this (comparatively simple) particular case are given in Appendix A. In section 4, we illustrate our theoretical results with numerical examples both for the Stokes equations and for the fluid/rigid particle motion problem.  $\phi$-FEM is also compared there with a standard \AL{(non isoparametric)} FEM on fitted meshes, demonstrating the superiority of $\phi$-FEM in terms of the accuracy achieved on comparable meshes. \AL{We end up with the last section giving some conclusions and perspectives.} 

\AL{Various notations for different parts of geometry and triangulation appear throughout the article. For readers' convenience, they are gathered in Appendix \ref{AppGlossary}.}

\section{Construction of the $\phi$-FEM scheme and main results}

\subsection{Governing equations}

We consider the motion of a viscous incompressible fluid around a solid (rigid) particle in the regime of creeping motion, i.e. neglecting all the inertial terms (for simplicity, we restrict ourselves here to the case of one particle, the extension to multiple particles being straight-forward). The particle is mobile and it moves under the action of the forces exerted by the surrounding fluid and the external forces (gravity). Let the fluid occupy (at a given time $t$) the domain $\Omega \subset \mathbb{R}^d$ ($d=2$ or $3$), the particle occupy the domain $\mathcal{S} \subset \mathbb{R}^d$, and denote $\mathcal{O}= \Omega \cup \bar{\mathcal{S}}$. Let  $\Gamma_w=\partial\mathcal{O}$ be the external boundary of the fluid domain (the immobile wall) where the fluid velocity is assumed to vanish,  $\Gamma=\partial\mathcal{S}$ be the fluid/solid interface, and assume that $\Gamma$ does not touch $\Gamma_w$, so that $\partial \Omega$ contains two disjoint components $\Gamma_w$ and $\Gamma$. For simplicity, we assume that the only external body force is gravitation with the constant acceleration $g$. Hence, the body force density in the fluid is $\rho_f g$ where $\rho_f$ is the constant
fluid density. Let $\rho_s$ be the constant density of the solid. Then, the
resultant external force on the particle is $mg$ where $m=\rho_s|\mathcal{S}|$ is the mass of the particle, and the resultant moment of the external force with respect to the barycenter of the particle is 0. Denoting the constant fluid viscosity by $\nu$, the equations governing the motion of the fluid/particle system can be now given as:
\begin{subequations}
\begin{align}
- 2\nu \Div D(u) + \nabla p &= \rho_f g, &&\tmop{ in } \Omega \text{}  \label{eq:1a}\\
 \Div u &= 0, &&\tmop{ in } \Omega\label{eq:div} \\
 \label{bc1}
 u &= U + \psi \times r, &&\tmop{ on } \Gamma \quad \\
 \label{bc2} u &= 0, &&\tmop{ on } \Gamma_w \\
\label{int u 1}
\int_{\Gamma} (2\nu D (u) - pI) n &= mg & &\\
\label{int u 2}
\int_{\Gamma} (2\nu D (u) - pI) n \times r &= 0 && \\
\label{intp0}
\int_\Omega p&=0 &&
\end{align}
\end{subequations}
Here, the unknowns are the fluid velocity $u:\Omega\to\mathbb{R}^d$ and the pressure $p:\Omega\to\mathbb{R}$, the velocity of the particle  barycenter $U\in \mathbb{R}^d$, and the angular velocity of the particle $\psi\in \mathbb{R}^{d'}$ ($d'=1$ if $d=2$ and $d'=3$ if $d=3$). In these equations, $D (u) = \frac{1}{2} (\nabla u + \nabla u^T)$ denotes the strain tensor, $r$ denotes the vector from the barycenter of the solid $\mathcal{S}$, and $n$ denotes the unit normal on $\Gamma$ looking into the solid. Equations (\ref{int u 1})--(\ref{int u 2})
come from the balance of forces exerted on the particle (the force exerted by
the fluid and the gravitational force). 
\phantom{\ref{eq:div}}

From a numerical simulation perspective, it is natural to introduce an immobile computational mesh on  the immobile box $\mathcal{O}$ containing both the fluid and the particle. On the other hand, the solid $\mathcal{S}$ will be moving with velocities $U=U(t)$, $\psi=\psi(t)$ at all-time $t$, thus permanently changing the shape of the fluid domain $\Omega$. It is therefore interesting to design numerical methods for the system (\ref{eq:1a})--(\ref{intp0}) that discretize $u$ and $p$ on a mesh non fitted to $\Omega$.

\subsection{A formal derivation of the appropriate weak formulation }

Let $\Th^\mathcal{O}$ be a regular simplicial mesh on $\mathcal{O}$ (the background mesh).  \AL{Assume that the solid and fluid domains are given by the level-set function $\phi$: $\mathcal{S}=\{\phi>0\}$ and $\Omega=\mathcal{O}\cap\{\phi<0\}$.} Introduce the active computational mesh $\mathcal{T}_h$ as a submesh  of $\Th^\mathcal{O}$ covering $\Omega$\AL{, i.e. excluding the cells of $\Th^\mathcal{O}$ lying completely inside $\mathcal{S}$}. Let $\Omega_h \supset \Omega$ be the domain of $\mathcal{T}_h$ and $G_h$ be the component  of $\partial\Omega_h$, other than $\Gamma_w$, \AL{and thus lying inside $\mathcal{S}$}, cf.~Fig.~\ref{fig:omega_h}.\footnote{\AL{In practice, the geometrical setting may be slightly more complicated. The rigorous theoretical definitions of $\Omega_h$ and $G_h$ will be given in (\ref{eq:def Th}) and \eqref{def G_h} and will be based on an approximation $\phi_h$ to the levelset $\phi$, rather than on $\phi$ itself. This may occasionally result in situations where some tiny portions of $\Omega$ lie outside $\Omega_h$ so that $G_h$ slightly penetrates $\Omega$. These technical details are not important for the forthcoming formal derivation of the FE scheme, while the rigorous proofs will be done assuming definitions (\ref{eq:def Th}) and \eqref{def G_h}. The actual implementation may introduce yet more geometrical approximations, as mentioned in Remark \ref{PractGeom}, which are not covered by our theory.}}

Assume (on a formal level, just to derive the scheme) that $u$ and $p$ can be extended from $\Omega$ to $\Omega_h$ as solution to the Stokes equations so that
\[ - 2 \nu \Div D (u) + \nabla p = \rho_f g \quad\text{ and }\quad\Div u = 0 \tmop{ in }
   \Omega_h . \]
Taking any sufficiently smooth test functions $v$ and $q$ on $\Omega_h$ such
that $v = 0$ on $\Gamma_w$, an integration by parts gives
\begin{equation}
  \label{weakOmhuv} 2 \nu \int_{\Omega_h} D (u) : D (v) - \int_{\Omega_h} p
  \Div v - \int_{\Omega_h} q  \Div u - \int_{G_h} (2 \nu D (u) -
  pI) n \cdot v = \int_{\Omega_h} \rho_f g \cdot v.
\end{equation}
Assuming $u=0$ on $\Gamma_w$, this imposes already the boundary condition (\ref{bc2}) on $\Gamma_w$, which we suppose to fit to the mesh $\Th$. On the contrary,  this formulation does not take into account any boundary conditions on
$\Gamma$.  In order to incorporate boundary conditions \eqref{bc1} we make the ansatz
\begin{equation}\label{ansatzu}
 u = \phi w + \chi (U + \psi \times r)
\end{equation}
where $\phi$ is the level-set for $\Omega$ so that $\phi=0$ on $\Gamma$, and  $\chi$ is a sufficiently smooth function on $\mathcal{O}$ such that $\chi= 1$ on the solid $\mathcal{S}$ and $\chi = 0$ on $\Gamma_w$.
This introduces the new vector valued unknown $w$ on $\Omega_h$ that should vanish on $\Gamma_w$ (indeed $w=0$ on $\Gamma_w$ implies $u=0$ on $\Gamma_w$ thanks to the choice of $\chi$; in fact, the reason for $\chi$ is to decouple the boundary conditions on $\Gamma$ and $\Gamma_w$ from one another).

The test functions $v$ in (\ref{weakOmhuv}) can be represented in the same way as the solution (\ref{ansatzu}):
\begin{equation}\label{ansatzv}
 v = \phi s + \chi (V + \omega \times r)
\end{equation}
for all vector-valued functions $s$ on $\Omega_h$ vanishing on $\Gamma_w$, and $V\in \mathbb{R}^d$, $\omega \in \mathbb{R}^{d'}$. In particular, the test functions of the form $\chi (V + \omega \times r)$ can be used  to take into account the force balance (\ref{int u 1}--\ref{int u 2}). To this end, we introduce  $B_h = \Omega_h \setminus \Omega$, i.e. the strip between $\Gamma$ and $G_h$, and use the divergence theorem on $B_h$ to
 transfer the boundary term in (\ref{weakOmhuv}) from $G_h$ to $\Gamma$ where it can be evaluated by (\ref{int u 1}--\ref{int u 2}):
\begin{multline}\label{transferGh1}
  \int_{G_h} (2\nu D (u ) - p  I) n \cdot \chi (V + \omega \times
  r) = \int_{G_h} (2\nu D (u ) - p  I) n \cdot (V + \omega \times
  r)\\
  = \int_{\Gamma} (2\nu D (u) - p  I) n \cdot (V + \omega \times r) +
  \int_{B_h} \tmop{div} (2\nu D (u ) - pI) n \cdot (V + \omega \times
  r)\\
  = m g \cdot V  - \int_{B_h} \rho_f g  \cdot \chi (V  + \omega \times
  r) =\cdots
\end{multline}
\MD{(the unit normal $n$ on $G_h$ in the first line is exterior with respect to domain $\Omega_h$, whereas $n$ on $\Gamma$ in the second line is the exterior unit normal with respect to domain $\Omega$, so that the exterior normals with respect to $B_h$ are $n$ on $G_h$ and $-n$ on $\Gamma$).}
We now remark $B_h=\Omega_h\setminus(\mathcal{O}\setminus\mathcal{S})$ to rewrite the above as
\begin{multline}\label{transferGh2}
  \cdots
  =   -\int_{\Omega_h} \rho_f g \cdot \chi (V + \omega \times r) +
     \int_{\mathcal{O}} \rho_f g \cdot \chi (V + \omega \times r) -
   \int_{\mathcal{S}} \rho_f g \cdot (V + \omega \times r) + mg \cdot V \\
  =   -\int_{\Omega_h} \rho_f g \cdot \chi (V + \omega \times r) +
  \int_{\mathcal{O}} \rho_f g \cdot \chi (V + \omega \times r) + \left( 1 -
   \frac{\rho_f}{\rho_s} \right) mg \cdot V.
\end{multline}
The last line is justified by observing $\int_{\mathcal{S}} \rho_s g \cdot (V + \omega \times r) = mg \cdot V$ with $\rho_s$ being the constant density of the solid. \AL{Indeed, $\int_{\mathcal{S}} \rho_s = m$ and $\int_{\mathcal{S}} \rho_s r = 0$ since $r=x-x_b$ is the vector pointing from the barycenter of the solid $x_b=\frac{1}{m}\int_{\mathcal{S}} \rho_s x$ to the current position $x$.}

Substituting the ansatzes (\ref{ansatzu})-(\ref{ansatzv}) for $u$ and $v$ into (\ref{weakOmhuv})  and rewriting the boundary term using (\ref{transferGh1})-(\ref{transferGh2}) we arrive at the following formal variational formulation of our problem in terms of the new unknowns $w,\vanessa{U},\psi$: find $w:\Omega_h\to\mathbb{R}^d$  vanishing on $\Gamma_w$, $U\in \mathbb{R}^d$, $\psi \in \mathbb{R}^{d'}$, and $p:\Omega_h\to\mathbb{R}$ such that
\begin{multline}\label{weakOmh}
  2\nu\int_{\Omega_h} D (\phi w+\chi (U + \psi \times r)) : D (\phi s + \chi (V + \omega \times r)) - \int_{G_h} (2\nu D (\phi w+\chi (U + \psi \times r)) - pI) n \cdot \phi s \\
  - \int_{\Omega_h} p \Div (\phi s+\chi (V + \omega \times r))
  - \int_{\Omega_h} q \Div (\phi w+\chi (U + \psi \times r))
  \\
 = \int_{\Omega_h} \rho_f g \cdot \phi s + \int_{\mathcal{O}} \rho_f g \cdot
   \chi (V + \omega \times r) + \left( 1 - \frac{\rho_f}{\rho_s} \right) mg
   \cdot V
\end{multline}
for all $s:\Omega_h\to\mathbb{R}^d$ vanishing on $\Gamma_w$, $V\in \mathbb{R}^d$, $\omega \in \mathbb{R}^{d'}$, and $q:\Omega_h\to\mathbb{R}$. In addition, the pressure $p$ should satisfy the constraint (\ref{intp0}).

Note that the formulation above  contains only the integrals on $\Omega_h$, $G_h$, $\mathcal{O}$ which can be easily approximated by quadrature rules on meshes $\Th$ and $\Th^\mathcal{O}$.
We can thus discretize using the usual finite elements for the trial and test functions.

\subsection{The $\phi$-FEM scheme: discretization with Taylor-Hood finite elements}

We fix an integer $k \geqslant 2$ and introduce the approximations $\phi_h$ and $\chi_h$ to the levelset $\phi$ and to the cut-off $\chi$, given by the standard nodal interpolation to the continuous FE spaces of degree $k$ on the mesh $\Th^\mathcal{O}$. The active computational mesh $\Th$, its domain $\Omega_h$ and the internal boundary component $G_h$ are actually defined as follows, cf. Fig. \ref{fig:omega_h},
\begin{equation}\label{eq:def Th}
\Th=\{T\in\Th^\mathcal{O}:T\cap\{\phi_h<0\}\neq\varnothing\},
\quad
\Omega_h:=\left(\cup_{T\in \mathcal{T}_h}T\right)^o \,,
\end{equation}
\begin{equation}\label{def G_h}
\MD{G_h=\partial\Omega_h\setminus\Gamma_w
=\{E \text{ (boundary facets of } \mathcal{T}_h)  \text{ such that } \AL{\phi_h\geq0\text{ on } E} \} \,.}
\end{equation}
Moreover, we shall need the collections of the mesh cells $\mathcal{T}_h^{\Gamma}$ and facets $\mathcal{F}_h^{\Gamma}$ near the boundary $\Gamma$, as illustrated in Fig. \ref{fig:penalty set}, to include the appropriate stabilization into the FE scheme. More specifically, we introduce the submesh $\Th^{\Gamma}\subset\Th$ and the corresponding subdomain $\Omega_h^{\Gamma}\subset\Omega$ containing the mesh elements intersected by the approximate interface $$\Gamma_h=\{\phi_h=0\},$$
i.e.
\begin{equation}\label{eq:def ThGamma}
\Th^{\Gamma}=\{T\in\Th:T\cap\Gamma_h\neq\varnothing\},
\quad
\Omega_h^{\Gamma}:=\left(\cup_{T\in \mathcal{T}_h^{\Gamma}}T\right)^o\,.
\end{equation}
Finally, we set $\mathcal{F}_h^{\Gamma}$ as  the collection of the interior facets of the mesh $\Th$ either cut by $\Gamma_h$ or  belonging to a cut mesh element
\begin{equation*}
  \mathcal{F}_h^{\Gamma} = \{E \text{ (an internal facet of } \mathcal{T}_h)
  \text{ such that } \exists T \in \mathcal{T}_h^\Gamma \text{ and } E \in \partial T\}.
\end{equation*}

\begin{figure}
    \centering
    \includegraphics[width=0.3\linewidth]{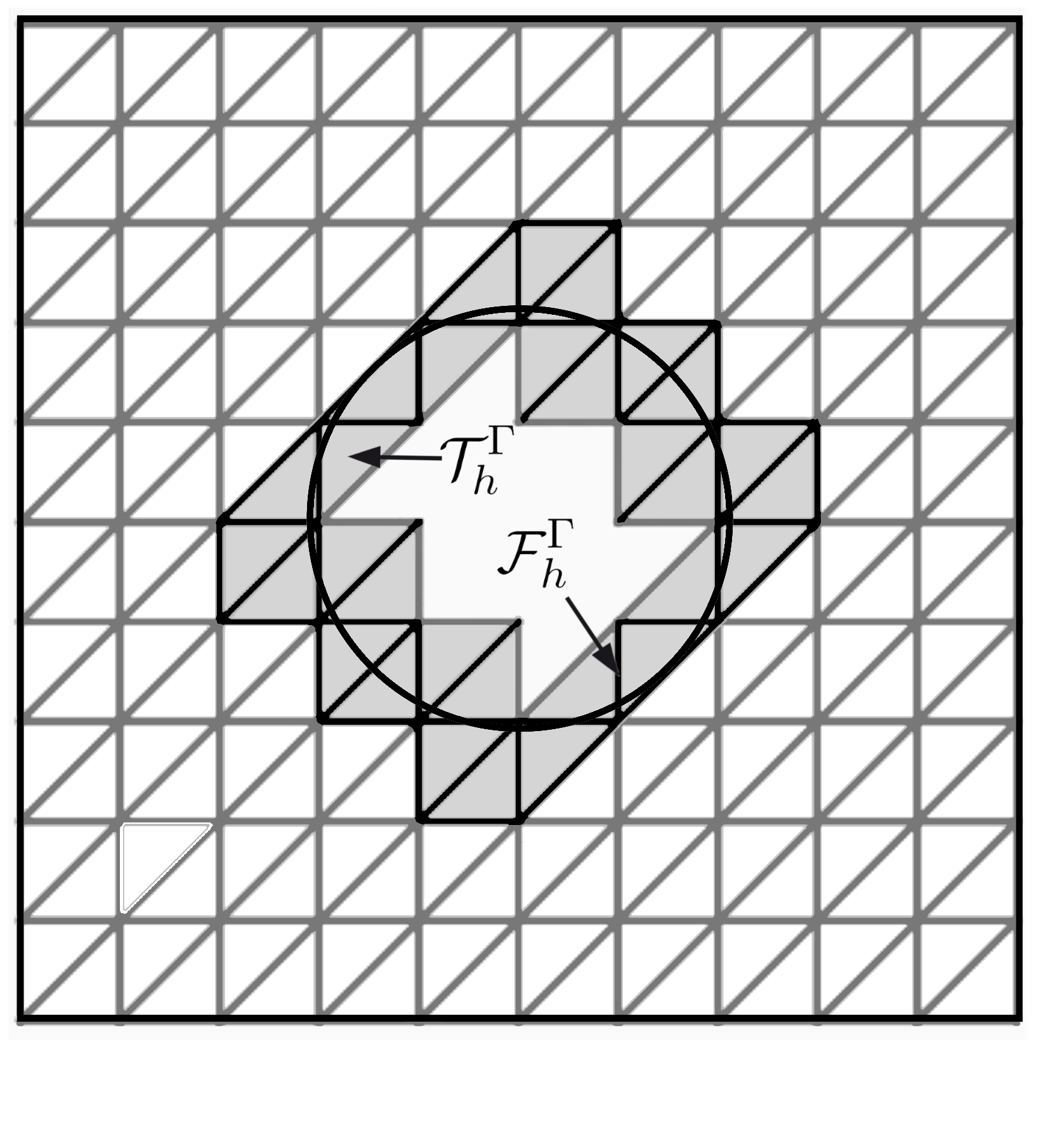}
    \caption{Example of $\mathcal{T}_h^{\Gamma}$ and $\mathcal{F}_h^{\Gamma}$ with the geometry given in Fig.~\ref{fig:omega_h} (left).}
    \label{fig:penalty set}
\end{figure}

\begin{rmk}\label{PractGeom} \MD{The definitions of $\Th$ and $\Th^\Gamma$ above assume an idealized setting where one can check the sign of $\phi_h$ at any point of any given mesh cell $T$. In practice, one would rather check this sign only at the vertices of the mesh or, eventually, at some other well chosen points. In our current implementation, we attribute the cells to $\Th$ or $\Th^\Gamma$ based on the sign of $\phi$ (equivalently, the sign of $\phi_h$) at the vertices only: in 2D, a triangle $T\in\Th^\mathcal{O}$ is selected to be in $\Th$ if $\phi\leqslant 0$ on at least one vertex of $T$; a triangle $T\in\Th$ is then selected to be in $\Th^\Gamma$ if $\phi\geqslant 0$ on at least one vertex of $T$. This deviation from definitions (\ref{eq:def Th})-(\ref{def G_h}) is not covered by our theory. 
}

\AL{We also note that in more advanced applications of $\phi$-FEM, $\phi_h$ and $\chi_h$ may be given directly on the discrete level, for instance by a discrete level-set equation. This possibility is however outside of the scope of the present article. We suppose here that the shape of the particle is sufficiently simple so that $\phi$ and $\chi$ are known analytically.}
\end{rmk}

Introduce the FE spaces for velocity and pressure on the mesh $\Th$:
$$ \mathcal V_h = \left\{ \nobracket v_h \in C (\bar{\Omega})^d : v_h |_T \in
   \mathbb{P}^k (T)^d \hspace{1em} \forall T \in \mathcal{T}_h, \hspace{1em}
   v_h = 0 \tmop{ on } \Gamma_w \right\}$$
   and
 $$ \mathcal M_h = \left\{ \nobracket q_h \in C (\bar{\Omega}) : q_h |_T \in
   \mathbb{P}^{k - 1} (T) \hspace{1em} \forall T \in \mathcal{T}_h,
   \hspace{1em} \int_{\Omega} q_h = 0 \right\} .
$$
\begin{rmk}\label{rmq1}
Note that \MD{the definition} of the pressure space involves an integral on $\Omega$, which is incompatible with our $\phi$-FEM framework since its whole point is to avoid integrals on $\Omega$ and $\Gamma$. In practice, we shall rather impose $\int_{\Omega_h} q_h = 0$, introducing a mismatch in the additive pressure constant (which, anyway, has no physical meaning) with respect to the exact solution satisfying (\ref{intp0}). We prefer however to keep the unimplementable constraint in the definition above to avoid some technical difficulties in theory. In practice, a special care will have to be taken in the interpretation of the error in pressure. We shall return to this technical point in the numerical results section.
\end{rmk}

The stabilized scheme inspired by (\ref{weakOmh}) can be now written as: find $w_h \in \mathcal V_h$, $U_h \in
\mathbb{R}^d, \psi_h \in \mathbb{R}^{d'}$, $p_h \in \mathcal M_h$ such that
\begin{multline}
  \label{sch} 2\nu\int_{\Omega_h} D (\phi_h w_h+\chi_h (U_h+\psi_h \times r) ) : D  (\phi_hs_h+\chi_h (V_h+\omega_h\times r) )\\
  - \int_{G_h} (2\nu D (\chi_h (U_h + \psi_h \times r) +\phi_h w_h)
  - p_hI) n \cdot \phi_h s_h
    - \int_{\Omega_h} p_h \Div (\phi_hs_h+\chi_h (V_h+\omega_h\times r) )  \\- \int_{\Omega_h} q_h \Div (\phi_h w_h+\chi_h (U_h+\psi_h \times r) ) \\
 + \sigma_u J_u(\chi_h (U_h + \psi_h \times r) +\phi_h w_h ,\chi_h (V_h + \omega_h \times r) +\phi_h s_h ) 
 \\
 + \sigma h^2  \sum_{T \in
   \mathcal{T}_h^{\Gamma}} \int_T (- \nu\Delta (\phi_h w_{h}+\chi_h (U_h+\psi_h\times r) ) + \nabla p_{h}) \cdot (-
   \nu\Delta (\phi_h s_h+\chi_h (V_h+\omega_h\times r) ) - \nabla q_h) \\
 + \sigma \sum_{T \in \mathcal{T}_h^{\Gamma}} \int_T \Div (\phi_h w_{h}+\chi_h (U_h+\psi_h\times r)) \Div (\phi_h s_h+\chi_h (V_h+\omega_h\times r) )  \\
 = \int_{\Omega_h} \rho_f g \cdot \phi_h s_h + \int_{\mathcal{O}} \rho_f g \cdot
   \chi_h (V_h + \omega_h \times r) + \left( 1 -  \frac{\rho_f}{\rho_s} \right) mg \cdot V_h \\
 + \sigma h^2  \sum_{T \in \mathcal{T}_h^{\Gamma}} \int_T \rho_f g \cdot (-
   \nu\Delta (\phi_h s_h+\chi_h (V_h+\omega_h\times r) ) - \nabla q_h)
\end{multline}
for all $s_h \in \mathcal V_h$, $V_h \in \mathbb{R}^d, \omega_h \in \mathbb{R}^{d'}$, $q_h \in \mathcal M_h$.

Here $J_u$ is the ghost penalties for the velocity, cf. \cite{ghost}:
\begin{align*} J_u(u,v) &=
 h \sum_{E \in \mathcal{F}_h^{\Gamma}} \int_E \left[
   {\partial_n} u \right] \cdot \left[
   {\partial_n} v \right]
   +h^3 \sum_{E \in \mathcal{F}_h^{\Gamma}} \int_E \left[
   {\partial_n^2} u \right] \cdot \left[
   {\partial_n^2} v \right]\,.
  \end{align*}
Note that, unlike \cite{ghost,burman}, we do not penalize the jumps of all the derivatives of the velocity; only the derivatives of order up to $2$ are included in $J_u$. There is no penalization on the pressure either. \MD{This alleviation of the ghost penalty is possible thanks to the additional least-squares-type stabilization (the terms multiplied by $\sigma$), cf. Lemmas \ref{lemma:poly} and \ref{LemDir:prop1 bis}. These least-squares terms are also necessary in themselves to control the fictitious extension of the solution outside $\Omega$, cf. the proof of Lemma \ref{lemma:coer}. Note that this extension is not present in CutFEM (this is indeed the principal difference between CutFEM and $\phi$-FEM).}
We also mention that the version of the ghost penalty in $\phi$-FEM for Poisson problem in \cite{phifem} is even more reduced: only the jumps of the first order derivatives are penalized there. The inclusion of the second order derivatives in $J_u$ in the present case of Stokes equations allows us to control both velocity and pressure in the forthcoming proofs, cf. Lemma \ref{lemma:poly}.

\subsection{Assumptions on the mesh and \MD{main results}}

Prior to stating our main results on the numerical convergence of our method, we begin with some geometrical assumptions on $\Omega$ and the functions $\phi$ and $\chi$.
\begin{assumption}\label{asm0}
The boundary $\Gamma$ can be covered by open sets $\mathcal{O}_i$, $i=1,\ldots,I$ and one can introduce on every $\mathcal{O}_i$ local coordinates $\xi_1,\ldots,\xi_d$ with $\xi_d=\phi$ such that all the partial derivatives $\partial^\alpha\xi/\partial x^\alpha$ and $\partial^\alpha x/\partial\xi^\alpha$ up to order $k+1$ are bounded by some $C_0>0$. Thus, $\phi$ is of class $C^{k+1}$ on $\mathcal{O}$.
\end{assumption}
\begin{assumption}\label{asm:chi}%
	$\chi\in H^{k+1}(\mathcal{O})$, $\chi=1$ on $\mathcal{S}$, $\chi=0$ on $\Gamma_w$.
\end{assumption}

We continue with assumptions on the mesh. To this end, we introduce an extended band of mesh elements near the boundary $\Gamma$, namely the submesh $\mathcal{T}_h^{\Gamma, ext}$ with  $\mathcal{T}_h^{\Gamma}\subset\mathcal{T}_h^{\Gamma, ext}\subset\Th$ by adding to $\mathcal{T}_h^{\Gamma}$ the cells which are neighbors  and neighbors of neighbors of cells in $\mathcal{T}_h^{\Gamma}$. 
	\begin{assumption}\label{asm1}%
		$|\nabla\phi|\ge m$, $|\nabla\phi_h|\ge \frac m 2$ on all the mesh cells in $\Th^{\Gamma,ext}$,  $|\phi|\ge mh$ on $\Th\setminus\Th^{\Gamma,ext}$, and $|\nabla\phi_h|\le M$ on $\Omega_h$ with some $m,M>0$.
	\end{assumption}

	\begin{assumption}\label{asm2}
		The approximate interface $G_h$ can be covered by element patch\-es $\{\Pi_k \}_{k = 1,
			\ldots, N_{\Pi}}$ having the following properties:
		\begin{itemize}
			\item Each $\Pi_k$ is composed of a mesh element $T_k$ lying inside $\Omega$ and some elements cut by $\Gamma$, more precisely $\Pi_k = T_k \cup \Pi_k^{\Gamma}$ where $T_k\in\mathcal{T}_h$, $T_k\subset\bar\Omega$, $\Pi_k^{\Gamma}\subset\mathcal{T}_h^{\Gamma}$, and $\Pi_k^{\Gamma}$ contains at most $N$ mesh elements;
			\item Each mesh element in a patch $\Pi_k$ shares at least a facet with another mesh element in the same patch. In particular, $T_k$ shares a facet $F_k$ with an element in $\Pi_k^\Gamma$;
			\item $\mathcal{T}_h^{\Gamma} = \cup_{k = 1}^{N_{\Pi}} \Pi_k^{\Gamma}$;
			\item $\Pi_k$ and $\Pi_l$ are disjoint if $k \neq l$.
		\end{itemize}
	\end{assumption}

\begin{assumption}\label{asmStokes}
    Any mesh cell $T\in\Th$ has at least $d$ facets not lying on $\Gamma_w$.
\end{assumption}
\MD{\begin{rmk}\label{RmkAsmPhi}
		Assumptions \ref{asm0}, \ref{asm1}, \ref{asm2} are similar to those made in the previous $\phi$-FEM publications \cite{phifem,phifem2}, which contain a more detailed discussion about them and some illustrations. In what concerns the mesh, these assumptions are satisfied if the mesh is sufficiently refined and $\Gamma$ is sufficiently smooth.  In what concerning the level-set function $\phi$,  we require essentially that it behaves like the signed distance to  $\Gamma$ near $\Gamma$ and it is bounded away from zero far from $\Gamma$, while remaining globally smooth. In general, one cannot thus take the signed distance to  $\Gamma$ as $\phi$ everywhere since it is guaranteed to be smooth only in a vicinity of $\Gamma$.
\end{rmk}}

\begin{rmk}
The last assumption \ref{asmStokes} is usually required in the theoretical analysis of Taylor-Hood elements for the Stokes equation in the geometrically conforming setting \cite{ern2013theory}, although it can be significantly relaxed, at least in the 2D setting \cite{Boffi09}. 
Note that this assumption only affects the mesh near the outer wall $\Gamma_w$, more particularly in the corners of $\mathcal{O}$, which we treat in the standard geometrically conforming manner anyway. It does not impose any further restriction on the active mesh $\Th$ near the interface $\Gamma$, where $\phi$-FEM is effectively employed.
\end{rmk}

Let us now state our main results:

\begin{theorem}\label{th1}
Suppose that Assumptions \ref{asm0}--\ref{asmStokes} hold true and the mesh $\mathcal{T}_h$ is quasi-uniform.
Let $(u,U,\psi,p)\in H^{k+1}(\Omega)^d\times\mathbb{R}^d\times\mathbb{R}^{d'}\times H^{k}(\Omega)$  be the solution to \eqref{eq:1a}-\eqref{intp0} and $(w_h,U_h,\psi_h,p_h)\in \mathcal V_h\times\mathbb{R}^d\times\mathbb{R}^{d'}\times \mathcal M_h$ be the solution to \eqref{sch}.  Denoting
\begin{equation*}
u_h:=\chi_h(U_h+\psi_h\times r)+\phi_h w_h
\end{equation*}
it holds for $h\le h_0$
\begin{equation}\label{H1err}
  | u - u_h|_{1, \Omega\cap\Omega_h}+ \AL{\frac{1}{\nu}}\|p - p_h\|_{0, \Omega\cap\Omega_h} \le Ch^k (\|u\|_{k+1,\Omega}+ \AL{\frac{1}{\nu}}\|p\|_{k,\Omega})
\end{equation}
and
\begin{equation}\label{rigerr}
|U-U_h| + |\psi-\psi_h| \le Ch^k (\|u\|_{k+1,\Omega}+ \AL{\frac{1}{\nu}}\|p\|_{k,\Omega})
\end{equation}
with some $C>0$ and $h_0>0$ depending on the parameters $C_0$, $m$, $M$, $N$  in Assumptions \ref{asm0}--\ref{asmStokes}, on the maximum of the derivatives of $\phi$ and $\chi$ of order up to $k+1$, on the mesh regularity, and on the polynomial degree $k$, but independent of $h$, $f$, and $u$.

Moreover, supposing $\Omega\subset\Omega_h$
\begin{equation}\label{L2err}
   \| u - u_h\|_{0, \Omega}
   \le Ch^{k+1/2}
 (\|u\|_{k+1,\Omega}+ \AL{\frac{1}{\nu}}\|p\|_{k,\Omega})
 \end{equation}
with a constant $C>0$ of the same type as above.
\end{theorem}
\begin{rmk}
The numerical results in Section \ref{NumSec} suggest that the convergence order for the particle velocity error and that for the $L^2$-error of the fluid velocity is $k+1$. This suggests that both estimates (\ref{rigerr}) and (\ref{L2err}) are not sharp. This is similar to our previous studies \cite{phifem} and \cite{phifem2}, \AL{in what concerns the $L^2$-error}.
\end{rmk}

\section{Proofs}

From now on, we put the viscosity of the fluid to $\nu=1$ to simplify the formulas. \MD{The general case can be easily recovered by dividing the governing equations by $\nu$ and redefining $\frac{p}{\nu}$ as $p$.} We shall also tacitly suppose that Assumptions \ref{asm0}--\ref{asmStokes} holds true.

This section is organized as follows:  we start  with some technical lemmas in Sections \ref{sec:lemma} and \ref{sec:inter}, essentially adapting the corresponding results from \cite{phifem}. Note however that the interpolation error bound in Section \ref{sec:inter} is sharper than its counterpart in \cite{phifem}; it is now optimal with respect to the Sobolev norm expected from the interpolated function. The proofs of Taylor-Hood inf-sup stability, the generalized coercivity of the bilinear form and finally the \textit{a priori} error estimates are then given, respectively, in Sections \ref{sec:taylor}, \ref{sec:coer} and \ref{sec:a priori}, thus establishing Theorem \ref{th1}.

\subsection{Some technical lemmas.}\label{sec:lemma}

\AL{Lemmas \ref{lemma:poly}, \ref{LemDir:prop1 bis}, and \ref{LemDir:prop2} are adaptions of, respectively, Lemmas 3.2, 3.3, and 3.4-3.5 from \cite{phifem}. Lemma \ref{LemKorn} is a version of the well known Korn inequality stating the uniformity of the constant in this inequality for a family of $h$-dependent domains $\Omega_h$. All these results, most notably Lemma \ref{LemDir:prop1 bis}, are necessary to prove the inf-sup stability of our scheme.}

 \begin{lemma}
   \label{lemma:poly}Let $T$ be a triangle/tetrahedron, $E$ one of its sides, $v$
   a vector-valued polynomial function on $T$, $q$ a scalar-valued polynomial
   function on $T$ such that
   \begin{equation}\label{eq:v q 0}
     v=  \frac{\partial v}{\partial n}= \frac{\partial^2 v}{\partial n^2}= 0 \tmop{ on } E
   \end{equation}
   and
     \begin{equation}\label{eq:delta v  0} - \Delta v + \nabla q = 0\mbox{ and }\Div v=0 \tmop{ on } T.
     \end{equation}
   Then $v = 0$ and $q=const$ on $T$.
 \end{lemma}

 \begin{proof}
 We shall give the proof only in the two dimensional setting, the generalization to the case $d=3$ being straightforward.  Without loss of generality, we can choose the Cartesian coordinates $(x,y)$ such that the edge $E$ lies on the $x$-axis. We shall denote the components of the vector-valued function $v$ by $(v_1,v_2)$.

 Let us write, for $k\in\{1,2\}$,
 $$v_k=\sum\limits_{\scriptstyle i,j\geq0} v_{ij}^kx^iy^j
 \mbox{ and }q=\sum\limits_{\scriptstyle i,j\geq0} q_{ij}x^iy^j.$$
 We will prove by strong induction on $m$ that
 \begin{equation}\tag{$S_m$}\label{eq:Sm}
v^k_{ij}=0 ~~~\forall k\in\{1,2\},~i\ge 0,~j\in \{0,...,m\}.
 \end{equation}
 Using \eqref{eq:v q 0}, it holds
 $$\sum\limits_{i\geq 0} v_{i0}^kx^i=\sum\limits_{i\geq 0} v_{i1}^kx^i=\sum\limits_{i\geq 0} v_{i2}^kx^i=0$$
 for all $x$ and $k\in\{1,2\}$.
 Hence $v_{i0}^k=v_{i1}^k=v_{i2}^k=0$ for all $i\geq 0$, $k\in\{1,2\}$
 and we obtain ($S_2$).
 Assume that for a given $m\geq 2$, (\ref{eq:Sm}) holds true.
 Thanks to \eqref{eq:delta v  0}, one has
 \begin{equation}\label{eq:lapldisc1}
 -(i+2)(i+1)v_{(i+2)j}^1-(j+2)(j+1)v_{i(j+2)}^1+(i+1)q_{(i+1)j}=0,
 \end{equation}
 \begin{equation}\label{eq:lapldisc2}
 -(i+2)(i+1)v_{(i+2)j}^2-(j+2)(j+1)v_{i(j+2)}^2+(j+1)q_{i(j+1)}=0
 \end{equation}
 and
  \begin{equation}\label{eq:div disc}
(i+1)v_{(i+1)j}^1+(j+1)v_{i(j+1)}^2=0.
 \end{equation}
 From \eqref{eq:lapldisc1} and \eqref{eq:lapldisc2}, for each $i,j$
 \begin{equation*}
 (i+2)v_{(i+2)(j+1)}^1+\frac{(j+3)(j+2)}{i+1}v_{i(j+3)}^1
 =\frac{(i+3)(i+2)}{j+1}v_{(i+3)j}^2+(j+2)v_{(i+1)(j+2)}^2.
 \end{equation*}
 The last equality for $j=m-2$ combined with ($S_m$) implies that
 $$v^1_{i(m+1)}=0 ~~~\forall i\ge 0.$$
 Relation \eqref{eq:div disc} for $j=m$ and ($S_{m}$) gives
  $$v^2_{i(m+1)}=0 ~~~\forall i\ge 0,$$
 which leads to ($S_{m+1}$).
 Thus $v=0$. This also implies $\nabla q=0$ on $T$ thanks to \eqref{eq:delta v  0}.
 \end{proof}

 \begin{lemma}
  \label{LemDir:prop1 bis}
  For any  $\beta>0$ and any integers $s,r \ge 1$ there exists $0 < \alpha < 1$   depending only on the mesh regularity and $s,r$ such that for any continuous vector-valued $\mathbb{P}^s$ FE function  $v_h$  on $\Th$ and any continuous scalar $\mathbb{P}^r$ FE function   $q_h$ it holds
  \begin{multline}\label{vhestOmGam}
    \|D(v_h)\|_{0, \Omega_h^{\Gamma}}^2 + (1-\alpha) h^2|q_h|_{1, \Omega_h^{\Gamma}}^2
    \le \alpha \|D(v_h)\|_{0, \Omega_h}^2 \\
    + \beta \left(h^2\|-\Delta v_h +\nabla q_h \|_{0, \Omega_h^{\Gamma}}^2+\|\Div v_h  \|_{0, \Omega_h^{\Gamma}}^2
    + \sum_{E \in \mathcal{F}_h^{\Gamma}}
    (h\|[\partial_nv_h]\|_{0,E}^2+h^3 \|[\partial_n^2v_h]\|_{0,E}^2)
   \right).
  \end{multline}
\end{lemma}
\begin{proof}
 Thanks to Assumption \ref{asm2}, the boundary $\Gamma$ can be covered by  patches $\{\Pi_k\}_{k=1,...,N_\Pi}$. Take $\beta>0$ and set
 \begin{equation}\label{alphmax}
     \alpha:=\max\limits_{\Pi_k,v_h,q_h}F(\Pi_k,v_h,q_h),
 \end{equation}
 where
 $$F(\Pi_k,v_h,q_h)=
 \dfrac{ \| D(v_h) \|_{0, \Pi_k^{\Gamma}}^2 {+h^2| q_h |_{1, \Pi_k^\Gamma}^2} - \beta G(\Pi_k,v_h,q_h)
 }{\| D(v_h) \|_{0, \Pi_k}^2{+h^2| q_h |_{1, \Pi_k^\Gamma}^2}}$$
 with
  $$G(\Pi_k,v_h,q_h)= h^2\|-\Delta v_h+\nabla q_h \|_{0,\Pi_k^{\Gamma}}^2+\|\Div v_h  \|_{0, \Pi_k^{\Gamma}}^2+h\left\|\left[\partial_n v_h\right]\right\|_{\mathcal{F}_k}^2+h^3\left\|\left[\partial_n^2 v_h\right]\right\|_{\mathcal{F}_k}^2.$$
  The maximum in (\ref{alphmax}) is taken over all the continuous vector-valued $\mathbb{P}^s$ FE functions  $v_h$  on $\Pi_k$, all the continuous scalar $\mathbb{P}^r$ FE functions  $q_h$   on $\Pi_k^\Gamma$, such that the denominator in the expression for $F$ does not vanish, and over all the possible configurations of patches $\Pi_k$ satisfying Assumption \ref{asm2}. The notation $\mathcal{F}_k$ stands for the set of mesh facets inside the patch $\Pi_k$ which includes thus $F_k$ separating $T_k$ from $\Pi_k^\Gamma$ and the other facets inside $\Pi_k^\Gamma$. The norm $\|\cdot\|_{\mathcal{F}_k}$ should be understood as $(\sum_{F\in\mathcal{F}_k}\|\cdot\|_F^2)^{1/2}$.

  Since the maximized function $F$ is invariant with respect to the transformation  $x\mapsto \frac{1}{h}x$, $v_h\mapsto \frac{1}{h}v_h$, $q_h\mapsto q_h$, we can assume that $h=1$ in (\ref{alphmax}). Furthermore $F(\Pi_k,v_h,q_h)=F(\Pi_k,\lambda v_h,\lambda q_h)$ for any $\lambda\neq0$. Hence the maximum (\ref{alphmax}) is attained since it can be taken over all admissible patches with $h=1$ and all $v_h,q_h$ such that $\|D(v_h)\|_{0, \Pi_k}^2{+h^2| q_h |_{1,\Pi_k^\Gamma}^2}=1$, forming the unit sphere in the finite dimensional space of all $(v_h,q_h)$ factored by rigid body motions on $\Pi_k$ and constants on $\Pi_k^\Gamma$. 

  Clearly $\alpha\leqslant 1$. Let us prove by contradiction that $\alpha<1$.
  Assume that $\alpha=1$.
  Consider the patch $\Pi_k$ (with $h=1$) and $v_h$, $q_h$ with $\|D(v_h)\|_{0, \Pi_k}^2{+| q_h |_{1,\Pi_k^\Gamma}^2}=1$ on which the maximum (\ref{alphmax}) is attained.  Then,
 \begin{equation*}
 \|D(v_h)\|_{0, T_k}^2 +\beta G(\Pi_k,v_h,q_h)=0,
 \end{equation*}
 since $\Pi_k=T_k\cup\Pi_k^\Gamma$.
  We deduce that  $v_h=  \frac{\partial v_h}{\partial n}= \frac{\partial^2 v_h}{\partial n^2}= 0$ on all the facets in $\mathcal{F}_k$   and    $ - \Delta v_h + \nabla q_h = 0$, $\Div(v_h)=0$ on all $T\in\Pi_k^\Gamma$. Moreover, $v_h$ is a rigid body motion on $T_k$. Let $v_h^{rbm}$ be the rigid body motion velocity on $\mathbb{R}^d$ coinciding with $v_h$ on $T_k$. Thanks to Lemma \ref{lemma:poly} applied to $v_h-v_h^{rbm}$ and to $q_h$ on the cells in $\Pi_k^\Gamma$ starting from the cell adjacent in $\Pi_k^{\Gamma}$ to $T_k$, we have $v_h=v_h^{rbm}$ on $\Pi_k$ and $q_h=const$ on $\Pi_k^\Gamma$ (recall that $q_h$ is continuous). We have thus reached a contradiction with the assumptions  $\|D(v_h)\|_{0, \Pi_k}^2{+| q_h |_{1,\Pi_k^\Gamma}^2}=1$ and $\alpha=1$.

This proves that there exists $\alpha<1$ such that
\begin{multline*}
\|D(v_h)\|_{0, \Pi_k^\Gamma}^2+(1-\alpha)h^2| q_h |_{1, \Pi_k^\Gamma}^2
\le \alpha\|D(v_h)\|_{0, \Pi_k}^2 \\
+\beta\left(
h^2\|-\Delta v_h+\nabla q_h \|_{0,\Pi_k^{\Gamma}}^2+\|\Div v_h  \|_{0, \Pi_k^{\Gamma}}^2+h\left\|\left[\partial_n v_h\right]\right\|_{\mathcal{F}_k}^2+h^3\left\|\left[\partial_n^2 v_h\right]\right\|_{\mathcal{F}_k}^2
\right)
\end{multline*}
on all the patches $\Pi_k$ and for all $v_h,q_h$. Summing this over all $\Pi_k$ gives (\ref{vhestOmGam}).
 \end{proof}

\begin{lemma}\label{LemKorn}
	For any $v \in H^1(\Omega)^d$ vanishing on $\Gamma_w$
	\begin{equation}\label{Korn}
	|v|_{1, \Omega_h} \le C \|D(v)\|_{0, \Omega_h} \,.
	\end{equation}
\end{lemma}
\begin{proof}

Since $v = 0$ on $\Gamma_{\omega}$, we have the following Korn inequality
\begin{equation}\label{KornOm}
 |v|_{1, \Omega } \le C \|D (v)\|_{0, \Omega }
\end{equation}
with a constant $C>0$ depending only on the shape of $\Omega$, cf. \cite[Theorem 6.3-4]{ciarlet88}. This
implies
\begin{equation}\label{korn3}
|v|_{1, \Omega_h^i} \le C \|D (v)\|_{0, \Omega_h},
\end{equation}
where $\Omega_h^i$ \vanessa{denotes} the mesh cells inside $\Omega$. Now, for any pair of mesh cells $T,T'$ sharing a facet $E$, we can prove
\begin{equation}\label{KornTT}
 |v|_{1, T } \le C \|D (v)\|_{0, T } + C | v |_{1, T' }
\end{equation}
with a constant $C$ independent of $h$.  Indeed, combining the Korn inequalities (\ref{KornOm}) and the trace theorem on the reference element leads to  $|v|_{1, T } \le C \|D (v)\|_{0, T } + C | v |_{1/2, E }$. Employing again the trace inequality
$| v |_{1/2, E } \le  C | v |_{1, T' }$ (see \cite[Lemma 7.5.26 ]{brenner2008mathematical})
leads to  (\ref{KornTT}).

 Let $\Omega_h^{i, 1}$ be $\Omega_h^i$ plus the cells which are not in
$\Omega_h^i$ but have a neighbor in $\Omega_h^i$.
For any such cell $T$, we
take $T'$ as its neighbor in $\Omega_h^i$, apply the
estimate above \MD{and sum which gives
$$ |v|_{1,  \Omega_h^{i, 1}\backslash \Omega_h^{i}} \le C \|D (v)\|_{0, \Omega_h^{i, 1}\backslash \Omega_h^{i} } + C | v |_{1, \Omega_h^{i} },$$
hence, using \eqref{korn3},}
\[ |v|_{1, \Omega_h^{i, 1}} \le C \|D (v)\|_{0, \Omega_h}. \]
Let $\Omega_h^{i, 2}$ be $\Omega_h^{i, 1}$ plus the cells which are not in
$\Omega_h^{i, 1}$ but have a neighbor in $\Omega_h^{i, 1}$. We have similar to
above
\[ |v|_{1, \Omega_h^{i, 2}} \le C \|D (v)\|_{0, \Omega_h} \]
and so on. After a finite number of steps, say $k$, we arrive at $\Omega_h^{i,
k} = \Omega_h$. And
\[ | v |_{1, \Omega_h} =|v|_{1, \Omega_h^{i, k}} \le C \|D
   (v)\|_{0, \Omega_h}. \]
\end{proof}

\begin{lemma}\label{LemDir:prop2}
    For any $s_h \in \mathcal V_h$ and any $V_h\in\mathbb{R}^d,\omega_h\in\mathbb{R}^{d'}$,
	\begin{equation*}
	\|\phi_h s_h\|_{0, \Omega_h^{\Gamma}} +\left(\sum_{F \in \mathcal{F}_h^{\Gamma}} h\|\phi_hs_h \|_{0, F}^2\right)^{1/2}
	+ \sqrt{h}\|\phi_hs_h \|_{0, G_h}\le Ch \|D(\phi_h s_h+\chi_h(V_h+\omega_h\times r))\|_{0, \Omega_h}
	\end{equation*}
	\AL{and
    \begin{equation}\label{BoundVhomh}
    |V_h | + | \omega_h | \leqslant C \|\phi_h s_h+\chi_h(V_h+\omega_h\times r) \|_{1, \Omega_h} \,.
    \end{equation}}
\end{lemma}
\begin{proof}
Take any $s_h \in \mathcal{V}_h,$ $V_h \in \mathbb{R}^d, ~ \omega_h \in \mathbb{R}^{d'}$ and denote $v_h = \phi_h s_h + \chi_h (V_h + \omega_h \times r)$. By \cite[Lemma 3.4]{phifem}
\begin{multline}   \label{fromPhiDir}
\| \phi_h s_h \|_{0, \Omega_h^{\Gamma}} \leqslant Ch |   \phi_h s_h |_{1, \Omega_h^{\Gamma}} \leqslant Ch (| v_h |_{1,   \Omega_h^{\Gamma}} + | \chi_h (V_h + \omega_h \times r) |_{1,   \Omega_h^{\Gamma}}) \\
\leqslant Ch (| v_h |_{1, \Omega_h^{\Gamma}} + \| \chi_h   \|_{1, \Omega_h^{\Gamma}} (| V_h | + | \omega_h |)) .
\end{multline}
By equivalence of norms
\begin{equation}\label{equiVhomh}
|V_h | + | \omega_h | \leqslant C \|V_h + \omega_h \times r\|_{0,   \Gamma} . \end{equation}
Denote by $B_h^{\Gamma}$ the band between $\Gamma$ and $\Gamma_h$. Applying the divergence theorem to the vector field $| V_h + \omega_h \times r |^2 {\nabla \phi}$ and noting that the normal on $\Gamma$ (resp. $\Gamma_h$) is given by $\pm \frac{\nabla \phi}{| \nabla \phi |}$ (resp. $\pm \frac{\nabla \phi_h}{| \nabla \phi_h |}$) gives
\[ \int_{\Gamma} | V_h + \omega_h \times r |^2| \nabla \phi | \leqslant \int_{\Gamma_h} | V_h    + \omega_h \times r |^2 \frac{|{\nabla \phi} \cdot\nabla \phi_h|}{| \nabla \phi_h |}
+ \left|\int_{B_h^{\Gamma}}    \Div \left( | V_h + \omega_h \times r |^2 {\nabla \phi} \right)\right|
\]
We now note that $|\nabla \phi|$ (resp. $| \nabla \phi_h |$ and $\chi_h$) are both positive and bounded away from 0 on $\Gamma$ (resp. on $\Gamma_h$) uniformly in $h$ for $h$ small enough, and the measure of $B_h^{\Gamma}$ is of order $h^{k + 1}$. The inequality above implies thus \[ \|V_h + \omega_h \times r\|_{0, \Gamma}^2 \leq C (\| \chi_h (V_h + \omega_h    \times r) \|^2_{0, \Gamma_h} + h^{k + 1}  (|V_h | + | \omega_h |)^2 ) . \] Combining this with (\ref{equiVhomh}) gives, for $h$ small enough \[ |V_h | + | \omega_h | \leqslant C \| \chi_h (V_h + \omega_h \times r)\|_{0,    \Gamma_h} = C \|v_h \|_{0, \Gamma_h} \] and, by the trace inequality,  $|V_h | + | \omega_h | \leqslant C \|v_h \|_{1, \Omega_h}$, \AL{i.e. (\ref{BoundVhomh})}.
Substituting this into (\ref{fromPhiDir}) and combining with the Korn inequality (\ref{Korn}) yields the announced estimate for $\| \phi_h s_h \|_{0, \Omega_h^{\Gamma}}$ since $\| \chi_h   \|_{1, \Omega_h^{\Gamma}}$ is bounded uniformly in $h$. The remaining part of the estimate follows by trace inverse inequalities as in \cite[Lemma 3.5]{phifem}.
\end{proof}

\subsection{Interpolation by finite elements multiplied with the level set }\label{sec:inter}

We recall first a Hardy-type inequality, cf. \cite{phifem}.
\begin{lemma}  \label{lemma:hardy}
For any integer $s\in[0,k]$ and any $u \in H^{s + 1} (\Omega_h)$ vanishing on $\Gamma$, it holds
$ \left\| \displaystyle{\frac{u}{\phi}} \right\|_{s, \Omega_h} \le C \| u \|_{s + 1, \Omega_h}$  with $C > 0$ depending only on the constants in Assumption \ref{asm0} and on $s$.
\end{lemma}

This allows us to prove the following bound for interpolation by the products of finite elements with $\phi_h$.
\begin{lemma}\label{LemInterp}  Let $t$ be an integer $1 \leqslant t \leqslant k + 1$. 
For any $v \in H^t  (\Omega_h)^d \cap H^1_0(\Omega)$ there exists $w_h \in \mathcal V_h$ s.t.
\begin{equation}    \label{eq:v-wh phih}
\| v - \phi_h w_h \|_{s, \Omega_h} \leqslant Ch^{t -    s} \| v \|_{t, \Omega_h}, \quad s = 0, 1
\end{equation}
with $C > 0$ depending only on $t,s$, the constants in Assumptions \ref{asm0}--\ref{asm1}, and the mesh  regularity.
\end{lemma}
\begin{proof}
Let $v \in H^t (\Omega_h)^d$, $v=0$ on $\Gamma$, and set $w = v / \phi$. Thanks to Lemma  \ref{lemma:hardy}, $w \in H^{t - 1} (\Omega_h)^d$ and $ \left\| w \right\|_{t-1, \Omega_h} \le C \| v \|_{t, \Omega_h}$. Consider $w_h =  I_h^c w$, where $I_h^c$ is a Scott-Zhang interpolation operator. For any $T\in\Th$,  let $\omega_T$ denote the patch of mesh cells adjacent to $T$ (not necessarily all  the adjacent cells) regrouping the cells affected by the construction of  $I_h^c w$ on $T$, so that $I_h^c w$ on $T$ depends on $w$ only \vanessa{through} its restriction to $\omega_T$. The Scott-Zhang interpolation operator  can be constructed so that  $\omega_T\subset\mathcal{T}_h^{\Gamma,ext}$ for all $T\in\mathcal{T}_h^{\Gamma,ext}$, and $\omega_T\subset\mathcal{T}_h \setminus \mathcal{T}_h^{\Gamma,ext}$ for all $T\in\mathcal{T}_h\setminus  \mathcal{T}_h^{\Gamma,ext}$.  In what follows, we assume that the operator $I_h^c$ enjoys this property together with the usual  interpolation error estimates see for instance \cite{brenner2008mathematical}.

Our first goal is to prove  (\ref{eq:v-wh  phih}) for $s = 0$.
Taking any $T \in \mathcal{T}_h$. Recall that $\phi$ is supposed to be of  class (at least) $C^t$ so that $\| \phi - \phi_h \|_{\infty, T}  \leqslant Ch^t$. Hence,
\begin{align}      \label{cas0}
\|v - \phi_h w_h \|_{0, T} & =  \| \phi w - \phi_h      w_h \|_{0, T}\\
	\notag      & \leqslant  \| \phi \|_{\infty, T} \|w - w_h \|_{0, T} + \|      \phi - \phi_h \|_{\infty, T} \|w_h \|_{0, T}\\
	\notag     & \leqslant  \| \phi \|_{\infty, T} \|w - w_h \|_{0, T} + Ch^t      \|w\|_{0, \omega_T}    .
\end{align}
To continue this proof, we distinguish two cases: the cells   $T \in \mathcal{T}_h^{\Gamma, ext}$ close to $\Gamma_h$ and the remaining cells, which are at the distance of at least order $h$ from $\Gamma_h$.
	\begin{itemize}
		\item[(i)] Consider $T \in \mathcal{T}_h^{\Gamma, ext}$.    We have $\| \phi  \|_{\infty, T} \leqslant Ch$ on these cells since they are at    the distance $\sim h$ from $\Gamma$. Noting that $\|w - w_h \|_{0, T} \leqslant Ch^{t - 1} | w  |_{t - 1, \omega_T}$    by the usual interpolation estimate, we derive from (\ref{cas0})
		\begin{equation}      \label{casi}
		\|v - \phi_hw_h \|_{0, T} \leqslant      Ch^t \left(| w  |_{t - 1, \omega_T} + \|w\|_{0, \omega_T}\right)\,.
		\end{equation}
		\item[(ii)] Now consider $T \in \mathcal{T}_h \setminus    \mathcal{T}_h^{\Gamma, ext}$. We note that $\phi$   does not vanish on $\omega_T$ for such $T$ (recall that $\omega_T \subset \mathcal{T}_h \setminus    \mathcal{T}_h^{\Gamma, ext}$), so that $w \in H^t (\omega_T)$    and, by (\ref{cas0}) and the usual approximation estimates,
		\begin{equation}      \label{casii0}
		\|v - \phi_hw_h \|_{0, T} \leqslant      Ch^t \left(\| \phi \|_{\infty, T} | w  |_{t, \omega_T}  + \|w\|_{0, \omega_T}\right)\,.  \end{equation}
		In order to bound $| w  |_{t, \omega_T}$ here, \AL{we recall the Leibniz rule  valid for any multi-index $\alpha \in \mathbb{N}^d$
		$$
		\partial^{\alpha} v =\partial^{\alpha} (\phi w)
=\sum_{\scriptsize{\begin{array}{c}        \beta  \in \mathbb{N}^d\\        \beta \leqslant \alpha      \end{array}}}
C_{\alpha}^{\beta} \,(\partial^{\beta} \phi) \,  (\partial^{\alpha - \beta} w)
		$$
		with binomial coefficients $C_{\alpha}^{\beta}$ depending only on the    multi-indices $\alpha$ and $\beta$ (this formula can be easily proven by induction on the length of $\alpha$) $\beta \leqslant \alpha$ means $\forall i=1,\dots, d, \beta_i\leqslant \alpha_i$. If $\alpha\not=0$, this can be rewritten, by separating the term with $\beta=0$ (note that $C^0_\alpha=1$) and dividing by $\phi$, as
		\begin{equation}      \label{derw} \partial^{\alpha} w = \frac{1}{\phi} \partial^{\alpha} v     \quad - \sum_{\scriptsize{\begin{array}{c}        \beta  \in \mathbb{N}^d\\        \beta \leqslant \alpha, \beta \neq 0      \end{array}}} C_{\alpha}^{\beta} \frac{\partial^{\beta} \phi}{\phi}      \partial^{\alpha - \beta} w    \end{equation}}
		Applying (\ref{derw}) to $w$ on $\omega_T$ gives    \[ |w|_{t, \omega_T} \leqslant \frac{1}{\min_{\omega_T} | \phi |}       (|v|_{t, \omega_T} +C\|w\|_{t - 1, \omega_T}) . \]    Hence, by (\ref{casii0}),
		\begin{equation} \label{casii1}
		\|v - \phi_hw_h \|_{0, T}    \leqslant Ch^t  \left(\frac{\| \phi \|_{\infty, T}}{\min_{\omega_T} |        \phi |} (|v|_{t, \omega_T} +\|w\|_{t - 1, \omega_T})+ \|w\|_{0, \omega_T}\right).
		\end{equation}
		Recall that  $\min_{\omega_T} | \phi | \geqslant mh$ by Assumption \ref{asm1}. The distance between any point on $T$ and any point on  $\omega_T$ is at most $2h$    so that    \[ \frac{\| \phi \|_{\infty, T}}{\min_{\omega_T} | \phi |} = 1 +       \frac{\max_{T} | \phi | - \min_{\omega_T} | \phi       |}{\min_{\omega_T} | \phi |} \leqslant 1 + \frac{2Mh}{mh}
		\leqslant 1 + \frac{2M}{m}\]
		with $M$ denoting an upper bound on $| \nabla \phi |$.  Substituting this into (\ref{casii1}) gives
		\begin{equation}        \label{casii}
		\|v - \phi_hw_h \|_{0, T}   \leqslant Ch^t (|v|_{t, \omega_T} +\|w\|_{t - 1, \omega_T}).
		\end{equation} 	\end{itemize}
Summing (\ref{casi}) over all the cells $T \in \mathcal{T}_h^{\Gamma,ext}$
and (\ref{casii}) over  all the remaining cells of mesh $\Th$ gives
\[ \|v - \phi_h w_h \|_{0, \Omega_h} \leqslant Ch^t (|v|_{t, \Omega_h} +\|w\|_{t-1, \Omega_h}) \,. \]
This yields  \eqref{eq:v-wh phih} with $s = 0$ thanks to the estimate $\| w \|_{t - 1,  \Omega_h} \leqslant \| v \|_{t, \Omega_h}$ given by Lemma \ref{lemma:hardy}.

Let us now prove \eqref{eq:v-wh phih} for $s = 1$. Introduce $v_h = I_h^c v$  so that
\[ |v - v_h |_{1, \Omega_h} \leqslant Ch^{t - 1} |v|_{t, \Omega_h}
\quad\tmop{and}\quad
\| v - v_h \|_{0, \Omega_h} \leqslant Ch^t |v|_{t, \Omega_h} .  \]
Then, combining the already proven estimate \eqref{eq:v-wh phih} for $s = 0$, the  inverse inequality, and the interpolation estimates above, we obtain
\begin{align*}       |v - \phi_h w_h |_{1, \Omega_h} & \leqslant |v_h - \phi_h {w_h}       |_{1, \Omega_h} + |v - v_h |_{1, \Omega_h}\\
& \leqslant \frac{C}{h} \|v_h - \phi_h w_h\|_{0, \Omega_h} + |v - v_h |_{1,       \Omega_h}\\
       & \leqslant \frac{C}{h} \|v  - \phi_h w_h  \|_{0, \Omega_h} +       \frac{C}{h} \| v - v_h \|_{0, \Omega_h} + |v - v_h |_{1, \Omega_h}\\
       & \leqslant Ch^{t - 1} \|v\|_{t, \Omega_h} .
 \end{align*}
\end{proof}

\subsection{An adaptation of the Taylor-Hood inf-sup stability to $\phi$-FEM}\label{sec:taylor}
In this section, we prove some velocity-pressure inf-sup conditions that will be \vanessa{used} to establish the generalized coercivity (inf-sup) for the full bilinear form in the next section. The proofs are inspired by \cite{guzman} \AL{and start from an auxiliary inf-sup condition of Lemma \ref{lemGuzman} with respect to an $h$-dependent norm for the pressure. The final result in Lemma \ref{lemma:magic} is weaker than the usual inf-sup since it does not provide the control of the $L^2$ norm of the pressure over the whole domain $\Omega_h$. It will be however sufficient for our purposes since the encumbering term $- Ch^2 |p_h |^2_{1, \Omega_h^{\Gamma}}$ will be controlled by the stabilization present in the scheme, cf. Lemma \ref{LemDir:prop1 bis} for this matter.}

\begin{lemma}\label{lemGuzman}
There exists an $h$-independent constant $C > 0$ such that $\forall p_h \in \mathcal M_h$, $\exists w_h  \in \mathcal V_h$ satisfying
\begin{equation}    \label{infsup1h}
h^2 |p_h |^2_{1, \Omega_h} - Ch^2 |p_h |^2_{1, \Omega_h^{\Gamma}}
\leqslant \int_{\Omega_h}    \nabla p_h \cdot (\phi_h w_h)
\quad    \tmop{and} \quad
| \phi_h w_h |_{1, \Omega_h} \leqslant Ch |p_h |_{1,    \Omega_h} .
\end{equation}
\end{lemma}
\begin{proof}  Let us introduce the submesh $\mathcal{T}_h^i =\mathcal{T}_h \setminus  \mathcal{T}_h^{\Gamma}$ \AL{so that $\phi_h\leqslant 0$ on $T\in\Th^i$.} Denote by $\mathcal{E}_h^i$ the set of the edges  of the mesh $\mathcal{T}_h^i$ including those shared with  $\mathcal{T}_h^{\Gamma}$, but excluding those lying on $\Gamma_w$. For any  edge $E \in \mathcal{E}_h^i$, let $t_E$ be the unit tangent vector to $E$  (any of two, but fixed), $x_E$ be the midpoint of $E$, $\omega (E)$ be the  set of the mesh cells sharing $E$, and $\psi_E$ be the piecewise quadratic  function such that $\psi_E (x_E) = 1$ and $\psi_E$ vanishes at all the other  edge midpoints and at all the nodes of $\Th$. Moreover, define for all $E  \in \mathcal{E}_h^i$,
\[ \bar{\phi}_E = \left\{\begin{array}{ll}       - h, & \text{if } E   \text{ belongs to a cell from } \Th^{\Gamma,ext},\\       \displaystyle{\frac{1}{|E|}  \int_E \phi_h}, & \text{otherwise.}     \end{array}\right.
\]
Take any $p_h \in \mathcal M_h$ and set $w_h \in \mathcal V_h$ as
\begin{equation}    \label{defwh}
w_h = \sum_{E \in \mathcal{E}^i_h} \frac{h^2}{\bar{\phi}_E}    \psi_E  (t_E \cdot \nabla p_h) t_E  .
\end{equation}
We have indeed $w_h \in \mathcal V_h$, since the pressure tangential derivative $t_E  \cdot \nabla p_h$ is a continuous piecewise polynomial of degree $\leqslant  \AL{k - 2}$ on $\omega (E)$ and $\psi_E$ is a continuous piecewise polynomial of  degree $2$, vanishing outside $\omega (E)$. Note also that $w_h = 0$ on  $\Gamma_w$ since $\mathcal{E}^i_h$ does not contain the edges lying on  $\Gamma_w$.

\AL{Take any $E\in\mathcal{E}_h^i$ and any cell $T \in \omega (E) \cap \mathcal{T}_h^i$.} We shall see that
\begin{equation}\label{phihphiE}
    \int_T \frac{\phi_h}{\bar{\phi}_E} \psi_E  |t_E \cdot \nabla p_h |^2 \ge     c \int_T |t_E \cdot \nabla p_h |^2 \,.
\end{equation}
Here and elsewhere, the constants $c > 0$ depend only  on the polynomial degree $k$, the shape regularity, and the parameters of Assumption \ref{asm1}. \AL{To prove (\ref{phihphiE}), we set $\tilde{\phi}_h= \frac{\phi_h}{\bar{\phi}_E}$ and note that $\tilde{\phi}_h\geqslant 0$ on $T$ since $\phi_h\leqslant 0$ on $T\in\Th^i$. To derive further properties of $\tilde{\phi}_h$ from Assumption \ref{asm1}, we consider 3 following cases with respect to the placement of $T$ and $E$ in the mesh (we recall that $T\in\Th^i$ in any case and $E$ is an edge belonging to $T$).
\begin{description}
\item[Case 1] $T\in\Th^{\Gamma,ext}$. We have then $\bar\phi_E=-h$ so that
\begin{equation}\label{EstCase1}
\tilde{\phi}_h\geqslant 0
\text{  and }h\left|\nabla{\tilde{\phi}_h}\right|\ge \frac{m}{2}\text{ on }T \,.
\end{equation}
\item[Case 2] $T\not\in\Th^{\Gamma,ext}$, $E$ is not shared with any cell from $\Th^{\Gamma,ext}$. Then $\bar\phi_h=\phi_h(x_E)$ for some point $x_E\in E$ and we have for any $x\in T$
$$
\frac{\bar\phi_E}{\phi_h(x)}=
\frac{\phi_h(x_E)}{\phi_h(x)}=
1+\frac{\phi_h(x_E)-\phi_h(x)}{\phi_h(x)}=
1+\frac{\nabla\phi_h(c)\cdot(x_E-x)}{\phi_h(x)}
\le 1+\frac{Mh}{mh}=\frac{m+M}{m}
$$
where $c\in T$ is a point on the segment connecting $x_E$ with $x$. This implies
\begin{equation}\label{EstCase2}
\tilde{\phi}_h \ge \frac{m}{m+M}\text{ on }T \,.
\end{equation}
\item[Case 3] $T\not\in\Th^{\Gamma,ext}$, $E$ is shared with a cell from $\Th^{\Gamma,ext}$. Then $\bar{\phi}_E=-h$ and $\frac{1}{|E|}\int_E\phi_h\le -c_1h$ with some $c_1>0$ depending only on the constants in Assumption \ref{asm1} (since the distance between $E$ and $\Gamma_h$, where $\phi_h$ vanishes, is of order $h$ and $|\nabla\phi_h|$ is bounded away from 0 on $\Th^{\Gamma,ext}$). Combining this with the arguments of the previous case, we arrive at
\begin{equation}\label{EstCase3}
\tilde{\phi}_h \ge \frac{c_1m}{m+M}\text{ on }T \,.
\end{equation}
\end{description}
Moreover, in all of the 3 cases above,
\begin{equation}\label{EstCaseAbove}
|\tilde{\phi}_h|\leqslant C\text{ on }T \,.
\end{equation}
with some $C>0$ depending only on the constants in Assumption \ref{asm1}. In case 1, this follows from the bound $|\nabla\phi_h|\le M$ and the fact that the maximal distance between $T$ and $\Gamma_h$ is of order $h$. In cases 2 and 3, (\ref{EstCaseAbove}) can be proven in the same way as (\ref{EstCase2}) and (\ref{EstCase3}).}

\AL{Inequality (\ref{phihphiE}) can be now proven setting
\begin{equation}\label{MinphihphiE}
c=\min_{T,\tilde\phi_h,q_h} \frac{\int_T {\tilde{\phi}_h} \psi_E  q_h^2}{\int_T q_h^2}
\end{equation}
where the minimum is taken over all the simplexes $T$ permitted by the mesh regularity, all the polynomials $\tilde{\phi}_h:= \frac{\phi_h}{\bar{\phi}_E}$ of degree $k$ satisfying (\ref{EstCaseAbove}) and either of (\ref{EstCase1})--(\ref{EstCase2})--(\ref{EstCase3}), and all the polynomials $q_h:=t_E \cdot \nabla p_h\not= 0$ of degree $k-2$. By homogeneity and rescaling, one can safely assume that $h=1$ and $\|q_h\|_{0,T}=1$. The sets of possible $T$, $\tilde\phi_h$, $q_h$ are bounded and closed, so that the minimum in (\ref{MinphihphiE}) is indeed attained and $c>0$. Indeed, either of (\ref{EstCase1})--(\ref{EstCase2})--(\ref{EstCase3}) excludes the possibility of $\tilde\phi_h$ vanishing everywhere on $T$. This concludes the proof of (\ref{phihphiE}).}

Thanks to (\ref{phihphiE}), we have,  setting $\Omega_h^i=\Omega_h\setminus\Omega_h^\Gamma$ and denoting by $\mathcal{E} (T)$ the set of edges of a cell $T$ excluding the  edges on $\Gamma_w$,
\[
\int_{\Omega_h^i} \phi_h w_h \cdot \nabla p_h = \sum_{E \in     \mathcal{E}_h^i} h^2  \int_{\omega (E) \cap \Omega_h^i}     \frac{\phi_h}{\bar{\phi}_E} \psi_E  |t_E \cdot \nabla p_h |^2 \ge c     \sum_{T \in \mathcal{T}_h^i} {\sum_{E \in \mathcal{E} (T)}}_{} h^2      \int_T |t_E \cdot \nabla p_h |^2 \,.
\]
Taking into account Assumption \ref{asmStokes}, we have by scaling and the equivalence of norms  on all $T \in \mathcal{T}_h^i$  \[ \sum_{E \in \mathcal{E} (T)}  \int_T |t_E \cdot \nabla p_h     |^2 \ge c |p_h |^2_{1, T} \, . \]
Hence,
\begin{equation}    \label{estwhint}
\int_{\Omega_h^i} \phi_h w_h \cdot \nabla p_h \geqslant    ch^2 |p_h |^2_{1, \Omega^i_h} .
\end{equation}

Now, on a mesh cell $T \in \mathcal{T}_h^{\Gamma}$ having common edges with  cells of $\mathcal{T}_h^i$, definition (\ref{defwh}) clearly gives  \[ \left| \int_T \phi_h w_h \cdot \nabla p_h \right| = \left| \sum_{E \in     \mathcal{E} (T) \cap \mathcal{E}_h^i} h   \int_T \phi_h \psi_E |t_E \cdot     \nabla p_h |^2 \right| \leqslant Ch^2 | p_h |^2_{1, T} \]  since $\phi_h$ is of order $h$ on such cells. Combining this with  (\ref{estwhint}) gives  \begin{align}    \notag
ch^2 |p_h |^2_{1, \Omega_h} &\leqslant \int_{\Omega_h}    \phi_h w_h \cdot \nabla p_h - \int_{\Omega_h^{\Gamma}} \phi_h w_h \cdot    \nabla p_h + ch^2 |p_h |^2_{1, \Omega^{\Gamma}_h} \\
&\leqslant \int_{\Omega_h} \phi_h w_h \cdot \nabla p_h + Ch^2 |p_h |^2_{1,     \Omega^{\Gamma}_h} .
\label{estphbelow}
\end{align}
We also have
\begin{multline}    \label{estphabove}
| \phi_h w_h |_{1, \Omega_h}^2 = \int_{\Omega_h} | (\nabla \phi_h) w_h +    \phi_h (\nabla w_h) |^2\\    \le C \sum_{T \in \mathcal{T}_h} \sum_{E \in \mathcal{E} (T) \cap    \mathcal{E}_h^i} h^4  \int_T \left( \frac{| \nabla \phi_h    |^2}{\bar{\phi}_E^2} | \nabla p_h |^2 + \frac{\phi_h^2}{\bar{\phi}_E^2} |    \nabla \psi_E |^2 | \nabla p_h |^2 + \frac{\phi_h^2}{\bar{\phi}_E^2} |    \nabla^2 p_h |^2 \right)\\    \le C \sum_{T \in \mathcal{T}_h} h^2 \int_T  | \nabla p_h |^2 = Ch^2 |p_h    |_{1, \Omega_h}^2 .
\end{multline}
We have used here the finite element inverse estimates on $\phi_h, \psi_E$,  $p_h$, and the uniform upper bound (\ref{EstCaseAbove}) on  $|\tilde{\phi}_h|=\frac{| \phi_h|}{| \bar{\phi}_E |}$.

Redefining $w_h$ as $w_h / c$ with the constant $c$ from (\ref{estphbelow})  shows \eqref{infsup1h} as a combination of (\ref{estphbelow}) and  (\ref{estphabove}).
\end{proof}

\begin{lemma}\label{lemma:magic}
  There exists an $h$-independent constant $C > 0$ such that $\forall p_h \in \mathcal M_h$ $\exists s_h \in \mathcal V_h$
  \begin{equation}    \label{infsup0h}
  \| p_h \|_{0, \Omega_h}^2 - Ch^2 |p_h |^2_{1, \Omega_h^{\Gamma}}
  \le \int_{\Omega_h} \nabla p_h
    \cdot (\phi_h s_h) 
    \quad\tmop{ and }\quad
    | \phi_h s_h |_{1,\Omega_h} \leqslant C \| p_h \|_{0, \Omega_h}.
  \end{equation}
\end{lemma}

\begin{proof}
  Taking $p_h\in \mathcal M_h$.
  By continuous velocity-pressure inf-sup (recall that $\int_\Omega p_h=0$), there exists $v \in H_0^1 (\Omega)$  s.t.
  \begin{equation*}
    \Div v = - p_h
    \tmop{ on } \Omega, \tmop{ and } \| v \|_{1, \Omega} \leqslant C \| p_h
    \|_{0, \Omega}\,.
  \end{equation*}
  Let $\tilde{v} \in H^1 (\Omega_h)$ be the extension of $v$ by 0 outside
  $\Omega$. Lemma \ref{LemInterp} with $t=1$ implies $\exists v_h \in \mathcal V_h $ s.t.
  \[  \| \tilde{v} - \phi_h v_h \|_{0, \Omega_h} \leqslant Ch \| \tilde{v}
     \|_{1, \Omega_h} \leqslant Ch \| p_h \|_{0, \Omega_{_h}}
     \quad\tmop{ and }\quad
     | \phi_h v_h |_{1, \Omega_h} \leqslant C \| \tilde{v} \|_{1,
     \Omega_h} \leqslant C \| p_h \|_{0, \Omega_{_h}}. \]
  Thus,
  \begin{multline}\label{phOmega}
      \| p_h \|_{0, \Omega}^2 = - \int_{\Omega} p_h \Div v
      = \int_{\Omega} \nabla p_h \cdot \tilde v =
     \int_{\Omega_h} \nabla p_h \cdot (\phi_h v_h) + \int_{\Omega_h} \nabla
     p_h \cdot (\tilde{v} - \phi_h v_h) \\
   \leqslant \int_{\Omega_h} \nabla p_h \cdot (\phi_h v_h) + Ch | p_h |_{1,
     \Omega_h} \| p_h \|_{0, \Omega_{_h}} .
  \end{multline}
  We have
    \begin{equation*}
      \|p_h \|_{0, \Omega_{h}}^2 \leqslant C_1(\| p_h \|_{0, \Omega}^2 + h^2| p_h|_{1, \Omega_h^\Gamma}^2). 
    \end{equation*}
\AL{This can be proven by an argument similar to that in Lemma \ref{LemDir:prop1 bis}: one can consider the maximum of $\frac{\|p_h \|_{0, \Pi_{k}}^2} {\| p_h \|_{0,T_k}^2 + h^2| p_h|_{1, \Pi_k^\Gamma}^2}$ taken over all the admissible patches $\Pi_k$, as in Assumption \ref{asm2}, and piecewise polynomials $p_h$, observe that this maximum is attained and takes the value $C_1>0$, and sum up over all the patches covering $\Omega_h^\Gamma$.} \\ 
We can thus pass from the norm on $\Omega$ to that on $\Omega_h$ in (\ref{phOmega}):
  \[
      \| p_h \|_{0, \Omega_h}^2    \leqslant C_1\int_{\Omega_h} \nabla p_h \cdot (\phi_h v_h) + Ch | p_h |_{1,\Omega_h} \| p_h \|_{0, \Omega_h}
      +C_1h^2| p_h|_{1, \Omega_h^\Gamma}^2.
  \]
  Hence, by Young inequality,
  \begin{multline*}
      \frac 12 \| p_h \|_{0, \Omega_h}^2    \leqslant C_1\int_{\Omega_h} \nabla p_h \cdot (\phi_h v_h) + C_2h^2 | p_h |_{1,\Omega_h}^2
      +C_1h^2| p_h|_{1, \Omega_h^\Gamma}^2
      \\
     \leqslant
     C_1\int_{\Omega_h} \nabla p_h \cdot (\phi_h v_h)
     + C_2 \int_{\Omega_h} \nabla p_h \cdot (\phi_h w_h)
      +Ch^2| p_h|_{1, \Omega_h^\Gamma}^2
  \end{multline*}
  with $w_h$ given by Lemma \ref{lemGuzman}. Thus,
  \[ \frac 12 \| p_h \|_{0, \Omega_h}^2 - Ch^2| p_h|_{1, \Omega_h^\Gamma}^2
  \leqslant \int_{\Omega_h} \nabla p_h \cdot   \phi_h (C_1v_h + C_2w_h)
  \]and\[
     | \phi_h (C_1v_h + C_2w_h) |_{1, \Omega_h}
     \leqslant C \| p_h \|_{0, \Omega_{_h}} \]
  since $| \phi_h w_h |_{1, \Omega_h}
  \leqslant Ch |p_h |_{1,\Omega_h}
  \leqslant C \|p_h\|_{0,    \Omega_h}$ by inverse finite element estimates.

  Setting $s_h = 2(C_1v_h + C_2w_h)$  proves (\ref{infsup0h}).
\end{proof}

\subsection{The generalized coercivity (the inf-sup condition) for the bilinear form.}\label{sec:coer}
To ease the forthcoming calculations, let us introduce the finite element space of velocities combining the rigid body motion on the approximate boundary and the contributions involving the level set:
\begin{equation}\label{Vhrbm}
\mathcal V_h^{rbm}=\{\chi_h (V_h + \omega_h \times r) +\phi_h s_h
\text{ with } s_h\in \mathcal V_h, V_h\in\mathbb{R}^d,  \omega_h\in\mathbb{R}^{d'} \}.
\end{equation}
In the scheme \eqref{sch}, we shall now combine the test functions $s_h,V_h,\omega_h$ into $v_h\in \mathcal V_h^{rbm}$ as in the definition above. Similarly, we shall combine the trial functions $w_h,U_h,\psi_h$ into $u_h\in \mathcal V_h^{rbm}$  setting $u_h=\chi_h (U_h + \psi_h \times r) +\phi_h w_h)$. Scheme \eqref{sch} can be then rewritten in the compact form: find $u_h\in \mathcal V_h^{rbm}$ and $p_h\in \mathcal M_h$ such that
\begin{equation}\label{schCmpct}
	c_h (u_h, p_h ; v_h, q_h)=L_h(v_h, q_h), \quad \forall v_h\in \mathcal V_h^{rbm},\ q_h\in \mathcal M_h\,,
\end{equation}
where the bilinear form $c_h$ is given by
\begin{multline*}
  c_h (u_h, p_h ; v_h, q_h)
    = 2\int_{\Omega_h} D(u_h) : D(v_h)
    - \int_{G_h} (2D(u_h) - p_hI)n \cdot \phi_h s_h \\
   - \int_{\Omega_h} q_h \Div u_h
  - \int_{\Omega_h} p_h \Div v_h \\
   + \sigma h^2  \sum_{T \in \mathcal{T}_h^{\Gamma}} \int_T (- \Delta u_h + \nabla p_h) \cdot     (- \Delta v_h - \nabla q_h)
  + \sigma \sum_{T \in \mathcal{T}_h^{\Gamma}} \int_T (\Div u_h) (\Div v_h) \\
   + \sigma_u h \sum_{E \in \mathcal{F}_h^{\Gamma}} \int_E  \left[
\partial_n u_h \right] \cdot \left[
\partial_n v_h \right]
+ \sigma_u h^3 \sum_{E \in \mathcal{F}_h^{\Gamma}} \int_E \left[
\partial_n^2 u_h \right] \cdot \left[
\partial_n^2 v_h \right]
\end{multline*}
and the linear form $L_h$ is given by
\begin{multline}\label{defLh}
L_h(v_h, q_h) = \int_{\Omega_h} \rho_f g \cdot \phi_h s_h + \int_{\mathcal{O}} \rho_f g \cdot \chi_h (V_h + \omega_h \times r)  + \left( 1 -
\frac{\rho_f}{\rho_s} \right) mg \cdot V_h \\
+ \sigma h^2  \sum_{T \in \mathcal{T}_h^{\Gamma}} \int_T \rho_f g \cdot (-\Delta v_h- \nabla q_h) .
\end{multline}
In both expressions above, $s_h$, $V_h$, and $\omega_h$ are related to $v_h\in \mathcal V_h^{rbm}$ as in (\ref{Vhrbm}).

\begin{lemma}
  \label{lemma:coer}
  Introduce the norm on \ $\mathcal V_h^{rbm} \times \mathcal M_h$
  \[ \interleave v_h, q_h \interleave_h \assign
   \left( |v_h|_{1, \Omega_h}^2 + \| q_h \|_{0, \Omega_h}^2  + h^2
     \sum_{T \in \mathcal{T}_h^{\Gamma}} \| - \Delta v_h + \nabla q_h
     \|_{0, T}^2 + J_u (v_h,v_h) \right)^{1/2}. \]
   The following inf-sup condition holds \AL{provided $\sigma$ and $\sigma_u$ are
   	sufficiently large}:
  \[ \forall(u_h, p_h) \in \mathcal V_h^{rbm}  \times \mathcal M_h \quad \exists (v_h, q_h) \in
     \mathcal V_h^{rbm} \times \mathcal M_h  \]
  such that
  \begin{equation}\label{infsupch}
   \frac{c_h (u_h, p_h ; v_h, q_h)}{\interleave v_h, q_h \interleave_h} \geqslant \theta \interleave u_h, p_h     \interleave_h
  \end{equation}
  with a constant $\theta > 0$ depending only on the mesh regularity.
\end{lemma}

\begin{proof}
 Let us take $p_h \in \mathcal M_h$ and $u_h=\chi_h (U_h + \psi_h \times r) +\phi_hw_h\in \mathcal V_h^{rbm}$ with  $w_h \in \mathcal V_h$, $U_h\in\mathbb{R}^d$ and $\psi_h\in\mathbb{R}^{d'}$.

\noindent\textbf{Step 1: controlling the velocity.}
By choosing $(u_h,p_h)$ as the trial function and $(u_h,-p_h)$ as the test function in the bilinear form $c_h$, we obtained:
  \begin{multline}   \label{chuhuh}
    c_h (u_h, p_h ; u_h, - p_h) = 2\int_{\Omega_h} | D( u_h) |^2 - \int_{G_h} (2D(u_h) - p_h I) n \cdot \phi_h w_h \\
   + \sigma_u h \sum_{E \in \mathcal{F}_h^{\Gamma}} \int_E \left| \left[
     \partial_n u_h \right] \right|^2
     + \sigma_u h^3 \sum_{E \in \mathcal{F}_h^{\Gamma}} \int_E \left| \left[
     \partial_n^2 u_h \right] \right|^2
     \\
     + \sigma h^2 \sum_{T \in \mathcal{T}_h^{\Gamma}} \int_T | - \Delta u_h + \nabla
     p_h |^2 + \sigma \sum_{T \in \mathcal{T}_h^{\Gamma}} \int_T | \Div u_h |^2  .
  \end{multline}

  Let $\tilde{B}_h$ be the strip between $\Gamma_h=\{\phi_h=0\}$ and $G_h$,
  \tmtextit{i.e.} $\tilde{B}_h = \{\phi_h > 0\} \cap \Omega_h$. Since $\phi_h w_h = 0$
  on $\Gamma_h$,
  \begin{multline*}
    \int_{G_h} D(u_h) n \cdot \phi_h w_h =
    \int_{\partial \tilde{B}_h} D(u_h) n \cdot \phi_h w_h\\
    = \sum_{T \in \mathcal{T}_h^{\Gamma}} \int_{\partial (\tilde{B}_h \cap T)}
    D(u_h) n_T \cdot \phi_h w_h - \sum_{T \in \mathcal{T}_h^{\Gamma}}
    \sum_{E \in \mathcal{F}_h^{cut} (T)} \int_{\tilde{B}_h \cap E} D(u_h) n_T \cdot \phi_h w_h,
  \end{multline*}
  where $\mathcal{T}_h^{\Gamma}$ is defined in \eqref{eq:def ThGamma},  $\mathcal{F}_h^{cut} (T)$ regroups the facets of a mesh element $T$ cut by  $\Gamma_h$, and $n_T$ is the unit normal pointing outside of $T$ on the boundary of a mesh cell $T$. Applying the divergence theorem to the integrals on ${\partial (\tilde{B}_h \cap T)}$ and regrouping the integrals on the facets gives
  \begin{multline*}
    \int_{G_h} D(u_h) n \cdot \phi_h w_h =
    \int_{\tilde{B}_h}  D(u_h) : D(\phi_h w_h) + \sum_{T \in \mathcal{T}_h^{\Gamma}}
    \int_{\tilde{B}_h \cap T} \Div D(u_h) \cdot \phi_h w_h \\- \sum_{E \in
    \mathcal{F}_h^{\Gamma}} \int_{E \cap \tilde{B}_h} [D(u_h) n] \cdot
    \phi_h w_h \\
=\int_{\tilde{B}_h} | D(u_h) |^2 + \sum_{T \in \mathcal{T}_h^{\Gamma}}
\int_{\tilde{B}_h \cap T} \Div D(u_h) \cdot \phi_h w_h - \sum_{E \in
	\mathcal{F}_h^{\Gamma}} \int_{E \cap \tilde{B}_h} [D(u_h) n] \cdot
\phi_h w_h,
  \end{multline*}
  since $u_h-\phi_h w_h$ is the velocity of a rigid motion on $\tilde{B}_h$.

  Similarly (and simpler)
  \[
    \int_{G_h} p_h n \cdot \phi_h w_h = \int_{\partial \tilde{B}_h} p_h
    n \cdot \phi_h w_h = \int_{\tilde{B}_h} p_h \Div u_h + \int_{\tilde{B}_h}
    \nabla p_h \cdot \phi_h w_h.
  \]

  Substituting this into (\ref{chuhuh}) and rewriting $2\Div D(u_h) - \nabla p_h=\Delta u_h - \nabla p_h +\nabla\Div u_h$ on the cells $T\in\Th^\Gamma$ yields
  \begin{multline*}
    c_h (u_h, p_h ; u_h, - p_h) = 2\int_{\Omega_h} | D(u_h) |^2 -
    2\int_{\tilde{B}_h} | D(u_h) |^2
    - \underbrace{\sum_{T \in \mathcal{T}_h^{\Gamma}}  \int_{\tilde{B}_h \cap T} (\Delta u_h - \nabla p_h) \cdot \phi_h w_h}_{\text{Young with $\varepsilon_1$}} \\
    + \underbrace{2\sum_{F \in \mathcal{F}_h^{\Gamma}} \int_{F \cap \tilde{B}_h} [D(u_h) n] \cdot \phi_h w_h}_{\text{Young with $\varepsilon_2$}}
    + \underbrace{\int_{\tilde{B}_h} p_h \Div u_h}_{\text{Young with $\varepsilon_3$}}
    - \underbrace{\sum_{T \in \mathcal{T}_h^{\Gamma}} \int_{\tilde{B}_h \cap T}(\nabla\Div u_h) \cdot \phi_h w_h}_{\text{Young with $\varepsilon_4$}}
    \\
 + \sigma_u h \sum_{E \in \mathcal{F}_h^{\Gamma}}\int_E |\left[\partial_n u_h \right]|^2
 + \sigma_u h^3 \sum_{E \in \mathcal{F}_h^{\Gamma}} \int_E \left| \left[\partial_n^2 u_h \right] \right|^2
          \\
      + \sigma h^2  \sum_{T \in \mathcal{T}_h^{\Gamma}} \int_T | \Delta u_h - \nabla p_h |^2 + \sigma
     \sum_{T \in \mathcal{T}_h^{\Gamma}} \int_T | \Div u_h |^2
   .       \end{multline*}
  \AL{Several terms above are marked with ``Young with $\varepsilon_i$" meaning that we are going to apply the Young inequality with some weights $\varepsilon_1,\ldots,\varepsilon_4>0$ (multiplied by the appropriate powers of $h$) to these terms. We recall that  Lemma \ref{LemDir:prop2} implies
  	$$
  	\sum_{T \in \mathcal{T}_h^{\Gamma}}  \frac{1}{h^2}\|\phi_h w_h\|_{\tilde{B}_h \cap T}^2 \le C\|D(u_h)\|_{0,\Omega_h}^2 \,, \quad
  	\sum_{F \in \mathcal{F}_h^{\Gamma}} \frac{1}{h}\|\phi_h w_h\|_{F \cap \tilde{B}_h}^2 \le C\|D(u_h)\|_{0,\Omega_h}^2 \,,
  	$$
  	which allows us to absorb the norms of $\phi_hw_h$ into the first term with $\| D(u_h)\|_{0, \Omega_h}$. We also use the inverse inequality $h\| \nabla\Div u_h \|_{0, \Omega_h^{\Gamma}}\le C\| \Div u_h \|_{0, \Omega_h^{\Gamma}}$. This yields}
  \begin{multline*} c_h (u_h, p_h ; u_h, - p_h) \geqslant
  \left( 2 - C   {\frac{\varepsilon_1+\varepsilon_2+\varepsilon_4}{2} }\right) \| D(u_h)\|_{0, \Omega_h}^2 - 2\| D(u_h) \|_{0, \Omega^{\Gamma}_h}^2 -
     \frac{\varepsilon_3}{2}  \| p_h \|_{0, \Omega_h^{\Gamma}} \\
   + h^2  \left( \sigma - \frac{1}{2 \varepsilon_1} \right) \sum_{T \in
     \mathcal{T}_h^{\Gamma}} \| \Delta u_h - \nabla p_h \|_{0,T}^2
 + \left( \sigma - \frac{1}{2 \varepsilon_3} - \frac{C}{\varepsilon_4}
 \right)  \| \Div u_h \|_{0, \Omega_h^{\Gamma}}^2\\
 + h     \left( \sigma_u - \frac{1}{2 \varepsilon_2} \right)  \sum_{E \in
     \mathcal{F}_h^{\Gamma}} \|\left[ \partial_n u_h \right]\|_{0,E}^2
     + \sigma_u h^3 \sum_{E \in \mathcal{F}_h^{\Gamma}} \left\| \left[
     \partial_n^2 u_h \right] \right\|_{0,E}^2
  \end{multline*}
 Thanks to Lemma \ref{LemDir:prop1 bis}, this can be further bounded as
  \begin{multline}\label{cboundVel}
  c_h (u_h, p_h ; u_h, - p_h) \geqslant \left( 2(1-\alpha) - C
  {\frac{\varepsilon_1+\varepsilon_2+\varepsilon_4}{2} }  \right)
\| D(u_h)\|_{0, \Omega_h}^2
+2(1-\alpha)h^2| p_h |_{1, \Omega_h^{\Gamma}}\\
-\frac{\varepsilon_3}{2}  \| p_h \|_{0, \Omega_h^{\Gamma}}
+ h^2  \left( \sigma - \frac{1}{2 \varepsilon_1} -2\beta\right) \sum_{T \in
	\mathcal{T}_h^{\Gamma}} \| \Delta u_h - \nabla p_h \|_{0,T}^2\\
+ \left( \sigma - \frac{1}{2 \varepsilon_3} - \frac{C}{\varepsilon_4} -2\beta
\right)  \| \Div u_h \|_{0, \Omega_h^{\Gamma}}^2
+ h     \left( \sigma_u - \frac{1}{2 \varepsilon_2} -2\beta\right)  \sum_{E \in
	\mathcal{F}_h^{\Gamma}} \|\left[ \partial_n u_h \right]\|_{0,E}^2\\
+ h^3(\sigma_u-2\beta)  \sum_{E \in \mathcal{F}_h^{\Gamma}} \left\| \left[
\partial_n^2 u_h \right] \right\|_{0,E}^2
\end{multline}
with some $\beta>0$ and $\alpha\in(0,1)$.

\noindent\textbf{Step 2: controlling the pressure.}
  Let now $s_h\in \mathcal V_h$ be the function given by Lemma \ref{lemma:magic} and set $v_h^p=\phi_h s_h$. Noting that
  \[ - \int_{\Omega_h} p_h \Div v_h^p + \int_{G_h} p_h n
     \cdot v_h^p = \int_{\Omega_h} \nabla p_h \cdot v_h^p
     \geqslant \| p_h \|_{0, \Omega_h}^2 - Ch^2 |p_h |^2_{1, \Omega_h^{\Gamma}} \]
  we get
  \begin{multline*}
  c_h (u_h, p_h ; v_h^p,0)
  \geqslant \| p_h \|_{0, \Omega_h}^2 - Ch^2 |p_h |^2_{1, \Omega_h^{\Gamma}}  + 2\int_{\Omega_h} D(u_h) : D(v_h^p)
 - 2\int_{G_h} D(u_h)n \cdot v_h^p \\
  + \sigma h^2  \sum_{T \in \mathcal{T}_h^{\Gamma}}
  \int_T (\Delta u_h - \nabla p_h) \cdot \Delta v_h^p
+ \sigma \sum_{T \in \mathcal{T}_h^{\Gamma}} \int_T (\Div u_h) (\Div v_h^p) \\
+ \sigma_u h \sum_{E \in \mathcal{F}_h^{\Gamma}} \int_E  \left[
\partial_n u_h \right] \cdot \left[
\partial_n v_h^p \right]
+ \sigma_u h^3 \sum_{E \in \mathcal{F}_h^{\Gamma}} \int_E \left[
\partial_n^2 u_h \right] \cdot \left[
\partial_n^2 v_h^p \right]  .
\end{multline*}
Recalling that $|v_h^p|_{1,\Omega_h}\leqslant C\|p_h\|_{0,\Omega_h}$, remarking that $\|v_h^p\|_{0,G_h}\le\frac{C}{\sqrt{h}}|v_h^p|_{1,\Omega_h^\Gamma}$, and applying Young and inverse inequalities allows us to conclude
  \begin{multline}\label{cboundPres}
c_h (u_h, p_h ; v_h^p,0)
\geqslant \frac{1}{2} \| p_h \|_{0, \Omega_h}^2 - Ch^2 |p_h |^2_{1, \Omega_h^{\Gamma}} - C \| D(u_h) \|_{0, \Omega_h}^2 \\
- C  \left(  \sigma^2h^2  \sum_{T \in \mathcal{T}_h^{\Gamma}} \| - \Delta u_h + \nabla p_h \|_{0, T}^2 + \sigma^2\sum_{T \in \mathcal{T}_h^{\Gamma}}\| \Div u_h \|_{0,T}^2
\right.\\
     \left.
     \sigma_u^2h \sum_{E \in \mathcal{F}_h^{\Gamma}}
     \left\| \left[ \partial_n u_h \right]
     \right\|_{0, E}^2 +\sigma_u^2h^3 \sum_{E \in \mathcal{F}_h^{\Gamma}}
     \left\| \left[ \partial_n^2 u_h \right]
     \right\|_{0, E}^2     \right) .
\end{multline}

\noindent\textbf{Step 3: combining the estimates.}
Multiply (\ref{cboundPres}) by $\lambda>0$ and add it to (\ref{cboundVel}). This gives
  \begin{multline*}
c_h (u_h, p_h ; u_h+\lambda v_h^p, - p_h) \geqslant \left( 2(1-\alpha) - C
\frac{\varepsilon_1+\varepsilon_2+\varepsilon_4}{2}  -C\lambda\right) \| D(u_h)\|_{0, \Omega_h}^2  \\
+\frac{\lambda-\varepsilon_3}{2}  \| p_h \|_{0, \Omega_h}^2
  + (2(1-\alpha)- C\lambda)h^2 |p_h |^2_{1, \Omega_h^{\Gamma}} \\
+ h^2  \left( \sigma - \frac{1}{2 \varepsilon_1} -2\beta -C\lambda\sigma^2\right) \sum_{T \in
	\mathcal{T}_h^{\Gamma}} \| \Delta u_h - \nabla p_h \|_{0,T}^2\\
+ \left( \sigma - \frac{1}{2 \varepsilon_3} - \frac{C}{\varepsilon_4} -2\beta -C\lambda\sigma^2
\right)  \| \Div u_h \|_{0, \Omega_h^{\Gamma}}^2\\
+ h     \left( \sigma_u - \frac{1}{2 \varepsilon_2} -2\beta -C\lambda\sigma_u^2\right)  \sum_{E \in
	\mathcal{F}_h^{\Gamma}} \|\left[ \partial_n u_h \right]\|_{0,E}^2
+ (\sigma_u-2\beta - C\lambda\sigma_u^2) h^3 \sum_{E \in \mathcal{F}_h^{\Gamma}} \left\| \left[
\partial_n^2 u_h \right] \right\|_{0,E}^2  .
\end{multline*}
Taking $\lambda,\varepsilon_1,\varepsilon_2,\varepsilon_4$ small enough, $\varepsilon_3 =  \lambda/2$ and $\sigma,\sigma_u$ big enough, and recalling Korn inequality \eqref{Korn}, this amounts to
  \[ c_h (u_h, p_h ; w_h + \lambda v_h^p, -p_h) \geqslant  c \interleave u_h, p_h
     \interleave_h^2 \]
  with some $c>0$. We  also have easily
  \[ \interleave w_h + \lambda v_h^p, -p_h \interleave_h \leqslant C \interleave
     w_h, p_h \interleave_h, \]
  hence the inf-sup estimate (\ref{infsupch}) with $v_h=u_h + \lambda v_h^p$, $q_h=-p_h$.
\end{proof}

\subsection{\textit{A priori} error estimates.}\label{sec:a priori}

In this section, we will prove Theorem \ref{th1} following the argumentation of \cite{phifem}, which is ameliorated since we require only the optimal regularity $H^{k+1}(\Omega)^d \times H^{k}(\Omega)$ for the velocity-pressure pair $(u,p)$ given by \eqref{eq:1a}-\eqref{intp0}.

\textbf{Proof of the $H^1$ a priori error estimate \eqref{H1err}:}
Let $(u,p)\in H^{k+1}(\Omega)^d \times H^{k}(\Omega)$ with $u=U+\psi\times r$ on $\Gamma$ be the solution to the continuous problem \eqref{eq:1a}-\eqref{intp0} and $(u_h,p_h)\in \mathcal V_h^{rbm}\times \mathcal M_h$ with $u_h=\phi_hw_h+\chi_h(U_h+\psi_h\times r)$ be the solution to the discrete problem \eqref{sch}.
Choose sufficiently smooth extension $\tilde u$ and $\tilde p$  of $u$ and $p$ on $\Omega_h$ such that $\tilde u=u$, $\tilde p=p$ on $\Omega$, and
$$\|\tilde u\|_{k+1,\Omega_h}\leqslant C\|u\|_{k+1,\Omega}, \quad
\|\tilde p\|_{k,\Omega_h}\leqslant C\|p\|_{k,\Omega}.$$
Applying Lemma \ref{LemInterp} to $\tilde u-\chi(U+\psi\times r)$, which vanishes on $\Gamma$ and on $\Gamma_w$, we see that there exists $\tilde w_h\in \mathcal V_h$ such that
\[\|\tilde u-\chi(U+\psi\times r)-\phi_h\tilde w_h\|_{1,\Omega_h}\leqslant Ch^k\|\tilde u-\chi(U+\psi\times r)\|_{k+1,\Omega_h}.\]
This allows us to introduce $\tilde u_h=\phi_h\tilde w_h+\chi_h(U+\psi\times r)\in \mathcal V_h^{rbm}$ satisfying
\begin{eqnarray}\notag
\|\tilde u-\tilde u_h\|_{1,\Omega_h}
&\leqslant&\|\tilde u-\chi(U+\psi\times r)-\phi_h\tilde w_h\|_{1,\Omega_h}+\|(\chi-\chi_h)(U+\psi\times r)\|_{1,\Omega_h}\\
\notag
&\leqslant& Ch^k(\|\tilde u\|_{k+1,\Omega_h}+\|\chi\|_{k+1,\Omega_h}(|U|+|\psi|))\\
&\leqslant& Ch^k\|\tilde u\|_{k+1,\Omega_h}  \leqslant Ch^k\|u\|_{k+1,\Omega},
\label{InterpU}
\end{eqnarray}
thanks to the standard interpolation of $\chi\in H^{k+1}(\Omega_h)$ and to the bounds $|U|,|\psi|\le C\|u\|_{1,\Omega}$ valid by the trace inequality (recall that $u=U+\psi\times r$ on $\Gamma$).

Similarly, $\|\tilde u-\tilde u_h\|_{0,\Omega_h}\leqslant Ch^{k+1}\|u\|_{k+1,\Omega}$.
We define moreover $\tilde p_h\in\mathcal{M}_h$ by the standard FE nodal interpolation  $\tilde p_h=I_h\tilde p$ such that
\begin{equation}\label{InterpP}
\|\tilde p-\tilde p_h\|_{0,\Omega_h}\leqslant Ch^{k}\|p\|_{k,\Omega}\,.
\end{equation}

Thanks to  Lemma \ref{lemma:coer}, $ \exists (v_h, q_h) \in
     \mathcal V_h^{rbm} \times \mathcal M_h$  such that
  \begin{equation}\label{Cea1}
     \interleave \tilde u_h-u_h, \tilde p_h-p_h  \interleave_h \leq C \dfrac{c_h (\tilde u_h - u_h, \tilde p_h - p_h ; v_h, q_h)}{\interleave v_h, q_h \interleave_h}.
  \end{equation}
We should now substitute $(\tilde{u}, \tilde{p})$ into the form $c_{h.}$
To this end, we introduce the fictitious right-hand sides $\tilde{F}$ and $\tilde{Q}$ on $\Omega_h$ so that
\[ - 2 \Div D (\tilde{u}) + \nabla \tilde{p} = \tilde{F} \text{ and }\Div \tilde{u} = \tilde{Q} \tmop{ in } \Omega_h . \]
We observe then, taking any $v_h=\phi_hs_h+\chi_h(V_h+\omega_h\times r)\in \mathcal V_h^{rbm}$, $q_h\in \mathcal M_h$,
\[
2 \int_{\Omega_h} D (\tilde{u }) : D (v_h) - \int_{\Omega_h}
\tilde{p}  \Div v_h - \int_{G_h} (2 D (\tilde{u }) - \tilde{p} I) n
\cdot v_h = \int_{\Omega_h} \tilde{F} \cdot v_h
\]
and, recalling $B_h = \Omega_h \setminus \Omega$,
\begin{multline*}
\int_{G_h} (2 D (\tilde{u} ) - \tilde{p}  I) n \cdot (V_h +
\omega_h \times r) = \int_{\Gamma} (2 D (\tilde{u}) - p  I) n \cdot (V_h + \omega_h
\times r)\\ + \int_{B_h} \tmop{div} (2 D (\tilde{u} ) - \tilde{p}  I)
n \cdot (V_h + \omega_h \times r)
= m g \cdot V_h - \int_{B_h} \tilde{F}  \cdot (V_h + \omega_h \times
r).
\end{multline*}
Hence,
\begin{multline*}
    c_h  (\tilde{u} , \tilde{p}  ; v_h, q_h)
= \int_{\Omega_h} \tilde{F}  \cdot v_h -
\int_{B_h} \tilde{F} \cdot (V_h + \omega_h \times r)  + mg \cdot V_h
- \int_{\Omega_h} q_h \Div\tilde{u} \\
+ \sigma h^2  \sum_{T \in \mathcal{T}_h^{\Gamma}} \int_T 
\tilde{F}\cdot (- 
\Delta v_h- \nabla q_h) + \sigma
\int_{\Omega_h^{\Gamma}} (\Div\tilde{u}) (\Div v_h).
\end{multline*}

Also note that the RHS (\ref{defLh}) of the scheme (\ref{schCmpct}) can be rewritten as
\begin{multline*}
L_h (v_h, q_h) = \int_{\Omega_h} \rho_f g \cdot v_h - \int_{{B_h} }
\rho_f g \cdot (V_h + \omega_h \times r) + mg \cdot V_h\\
+ \sigma h^2  \sum_{T \in \mathcal{T}_h^{\Gamma}} \int_T \rho_f g \cdot (- 
\Delta v_h - \nabla q_h)
\end{multline*}
This allows us to establish the following Galerkin orthogonality relation, valid for all $v_h\in \mathcal V_h^{rbm}$, $q_h\in \mathcal M_h$,
\begin{equation}\label{GalOrt}
    c_h  (\tilde{u} - u_h, \tilde{p} - p_h ; v_h, q_h) = R_h (v_h, q_h),
\end{equation}
where
\begin{multline*}
R_h (v_h, q_h) = \int_{B_h} (\tilde{F}  - \rho_f g) \cdot \phi_hs_h
- \int_{B_h} q_h \Div\tilde{u} \\
+ \sigma h^2  \sum_{T \in \mathcal{T}_h^{\Gamma}} \int_T (\tilde{F}  -
\rho_f g) \cdot (- 
\Delta v_h- \nabla q_h) + \sigma
\int_{\Omega_h^{\Gamma}} (\Div\tilde{u}) (\Div v_h).
\end{multline*}
The integrals of $\tilde{F}  - \rho_f g$ and $\Div\tilde{u}$ on $\Omega_h$ have been rewritten as integrals on $B_h$  since both $\tilde{F}  - \rho_f g$ and $\Div\tilde{u}$ vanish on $\Omega$.

Combination of (\ref{Cea1}) and (\ref{GalOrt}) entails
\[
\interleave \tilde{u}_h - u_h, \tilde{p}_h - p_h \interleave_h \leq  C
\dfrac{c_h  (\tilde{u}_h - \tilde{u}, \tilde{p}_h - \tilde{p} ; v_h, q_h)
	+ R_h (v_h, q_h)}{\interleave v_h, q_h \interleave_h} .
\]
We can now use interpolation inequalities as in
{\cite[Section 3.4]{phifem}}. In particular, the  term with $c_h$ in the nominator of the fraction above is bounded by $Ch^k(\|u\|_{k+1,\Omega}+\|p\|_{k,\Omega}){\interleave v_h, q_h \interleave_h}$ thanks to (\ref{InterpU})--(\ref{InterpP}) and to the estimates of Lemma \ref{LemDir:prop2}. To bound $R_h(v_h,q_h)$
we recall that $\tilde{F}  - \rho_f g$ and $\Div\tilde{u}$ vanish on $\Omega$. Thus, thanks to \cite[Lemma 3.6]{phifem}
\begin{equation*}
    \|\tilde{F}  - \rho_f g\|_{0,\Omega_h^\Gamma}
    \le Ch^{k-1} \|\tilde{F}  - \rho_f g\|_{k-1,\Omega_h^\Gamma}
    \le Ch^{k-1} (\|u\|_{k+1,\Omega}+\|p\|_{k,\Omega})
\end{equation*}
and
\begin{equation}\label{bounddivuhstrip}
    \|\Div\tilde{u}\|_{0,\Omega_h^\Gamma}
    \le Ch^{k} \|\Div\tilde u\|_{k,\Omega_h^\Gamma}
    \le Ch^{k} \|u\|_{k+1,\Omega}.
\end{equation}
This, combined with the estimates of Lemma \ref{LemDir:prop2}, in particular $\|\phi_hs_h\|_{0,\Omega_h^\Gamma}\le Ch|v_h|_{1,\Omega_h^\Gamma}$, leads to $|R_h(v_h,q_h)|\leq Ch^k(\|u\|_{k+1,\Omega}+\|p\|_{k,\Omega}){\interleave v_h, q_h \interleave_h}$ and
\AL{\begin{equation}\label{H1exterr}
\interleave \tilde{u}_h - u_h, \tilde{p}_h - p_h \interleave_h \leq  Ch^k
(\|u\|_{k+1,\Omega}+\|p\|_{k,\Omega})\,.
\end{equation} }
Recalling again the interpolation estimates (\ref{InterpU})--(\ref{InterpP}), we obtain the error estimates for $u$ and $p$, announced by \eqref{H1err}.

\textbf{Proof of the a priori error estimate \eqref{rigerr} on the velocity of the solid:}
\AL{We have by the construction of the interpolant $\tilde{u}_h=\phi_h\tilde w_h+\chi_h(U+\psi\times r)$ and thanks to  (\ref{BoundVhomh})
\begin{equation*}
      |U-U_h|+|\psi-\psi_h|
\leqslant C\|\phi_h(\tilde w_h-w_h) +\chi_h(U-U_h+(\psi-\psi_h)\times r)\|_{1,\Omega_h}
=C\|\tilde u_h-u_h\|_{1,\Omega_h}
\end{equation*}
which proves \eqref{rigerr} thanks to (\ref{H1exterr}).
}

 \textbf{Proof of the $L^2$ a priori error estimate \eqref{L2err}:}
  Let $(v,q,V,\omega)\in H^2(\Omega)^d\times H^1(\Omega) \times\mathbb{R}^d\times\mathbb{R}^{d'}$ the solution to
\begin{equation*}
\left\{\begin{array}{ll}
- 2 \Div D(v) + \nabla q = u-u_h, &\tmop{ in } \Omega, 
\\
 \Div v = 0, &\tmop{ in } \Omega,
 \\
 v = V + \omega \times r, &\tmop{ on } \Gamma , \\
 v = 0, &\tmop{ on } \Gamma_w, \\[1pt]
\int_{\Gamma} (2 D (v) - q I) n = 0,  &\\[3pt]
\int_{\Gamma} (2 D (v) - q I) n \times r = 0, & \\[3pt]
\int_\Omega q =0. &\end{array}\right.
\end{equation*}
An integration by parts gives
\begin{align}\label{eq:esti H1 norm}
\|u-u_h\|_{0,\Omega}^2&=
\int_{\Omega}(u-u_h)(- 2\Div D(v)+\nabla q)\\
&=2\int_{\Omega}D(u-u_h):D(v) - \int_{\Omega}q\Div(u-u_h)
- \int_{\Omega}(p-p_h)\Div v .
\notag \\
&\qquad + \AL{\int_{\Gamma}((\phi_h-\phi)w_h+(\chi_h-\chi)(U_h+\phi_h\times r))\cdot (2D(v)-q I)n }
\notag
\end{align}
Note that the boundary term \AL{$\int_{\Gamma}u\cdot (2D(v)-q I)n$ vanishes since $u$ is a rigid body motion on $\Gamma$. For the same reason, $\int_{\Gamma}(\phi w_h+\chi(U_h+\phi_h\times r)\cdot (2D(v)-q I)n $ vanishes.}

Let $(\tilde v,\tilde q)\in H^2(\Omega_h)^d\times H^1(\Omega_h)$ coincide with $(v,q)$ on $\Omega$. They can be constructed by a bounded extension operator in $H^2\times H^1$ so that
\begin{equation}\label{regvq}
  \|\tilde v\|_{2,\Omega_h} +\|\tilde q\|_{1,\Omega_h}\leqslant
  C(\|v\|_{2,\Omega} +\|q\|_{1,\Omega})\leqslant
  C\|u-u_h\|_{0,\Omega}.
\end{equation}
We now further rewrite  \eqref{eq:esti H1 norm}
using Galerkin orthogonality (\ref{GalOrt}) with the test functions $v_h=\phi_hs_h+\chi_h(V_h+\omega_h\times r)\in \mathcal V_h^{rbm}$ and $q_h\in \mathcal M_h$ and recalling $B_h=\Omega_h\backslash\Omega$,
\begin{multline}\label{cf 3.24 phiDir}
   \|u-u_h\|_{0,\Omega}^2\\
= \underbrace{
\int_{\Omega_h} 2 D (\tilde u-u_h) : D (\tilde v-v_h)
-\int_{\Omega_h} (\tilde q-q_h) \Div (\tilde u -u_h)
-\int_{\Omega_h} (\tilde p-p_h) \Div (\tilde v -v_h) }
\limits_{I} \\
- \underbrace{\left(\int_{B_h} 2 D (\tilde  u-u_h) : D(\tilde v)
-\int_{B_h} \tilde q \Div (\tilde u -u_h)
-\int_{B_h} (\tilde p-p_h) \Div\tilde v \right)}
\limits_{II}  \\
  + \underbrace{\int_{G_h} (2 D (\tilde u-u_h)
  - (\tilde p-p_h)I) n \cdot \phi_hs_h}\limits_{III}
 \end{multline}\begin{multline*}
 \underbrace{- \sigma_u  h \sum_{E \in \mathcal{F}_h^{\Gamma}} \int_E \left[
  \frac{\partial}{\partial n} (\tilde u-u_h) \right] \cdot \left[
  \frac{\partial v_h}{\partial n} \right]
  -\sigma_uh^3 \sum_{E \in \mathcal{F}_h^{\Gamma}} \int_E \left[
  \frac{\partial^2}{\partial n^2} (\tilde u-u_h) \right] \cdot \left[
  \frac{\partial^2 v_h}{\partial n^2}  \right]}\limits_{IV}\\
-\underbrace{ \sigma h^2  \sum_{T \in
  \mathcal{T}_h^{\Gamma}} \int_T- \Delta ( \tilde u-u_h) + \nabla(\tilde p -p_{h})) \cdot (-   \Delta v_h - \nabla q_h) }\limits_{V}\\
- \underbrace{ \sigma \sum_{T \in \mathcal{T}_h^{\Gamma}} \int_T \Div (\tilde u-u_h) \Div v_h  }\limits_{VI}
+ \underbrace{ R_h( v_h,q_h)  }\limits_{VII} \\
+ \AL{\underbrace{\int_{\Gamma}((\phi_h-\phi)w_h+(\chi_h-\chi)(U_h+\phi_h\times r))\cdot (2D(v)-q I)n}\limits_{VIII} }
\end{multline*}

We now take $V_h=V,\omega_h=\omega$ and set $s_h\in \mathcal V_h$ so that $\phi_hs_h$ is an optimal interpolant of $v-\chi_h(V+\omega\times r)$, as guaranteed by Lemma \ref{LemInterp}. We also set $q_h=\tilde I_h \tilde q$ using an appropriate Clément interpolation $\tilde I_h$. We can now estimate all the terms of (\ref{cf 3.24 phiDir}) using the already proven estimate (\ref{H1exterr}) and the interpolation estimates for $\tilde v-v_h$ and $\tilde p-p_h$ . This gives
\begin{equation}\label{L2errpresque}
\|u-u_h\|_{0,\Omega}^2 \le
Ch^{k+1/2} (\|u\|_{k+1,\Omega}+ \|p\|_{k,\Omega})(\|\tilde v\|_{2,\Omega_h}+\|\tilde q\|_{1,\Omega_h}).
\end{equation}
\AL{In particular, term $I$ is completely standard and gives in fact a contribution of the optimal order $h^{k+1}$.}
Rather than go to the details of the tedious calculations leading to the bounds of the remaining terms, we prefer here to refer to the similar arguments used in \cite{phifem} to estimate the terms in eq.~(3.24). Indeed, the terms $II-III$ in (\ref{cf 3.24 phiDir}) can be treated as the terms $II-III$  in eq.~(3.24) of \cite{phifem}. Terms $IV-V$ in (\ref{cf 3.24 phiDir}) can be treated as term $IV$  in eq.~(3.24) of \cite{phifem}.  Terms $VI$ in (\ref{cf 3.24 phiDir}) is also similar to the latter (note, in particular,   $\|\Div\tilde{v}\|_{0,\Omega_h^\Gamma}    \le Ch \|v\|_{2,\Omega}$ similarly to (\ref{bounddivuhstrip})). Finally, term $VII$ in (\ref{cf 3.24 phiDir}) can be treated as term $V$  in eq.~(3.24) of \cite{phifem}. \MD{As in \cite{phifem}, all these terms result in the sub-optimal estimate of order $O (h^{k + 1 / 2})$. The origin of this sub-optimality lies in the lack of adjoint consistency in formulation (\ref{schCmpct}): the adjoint discrete problem cannot be interpreted as a consistent discretization of a meaningful continuous problem. }

\AL{The only term in (\ref{cf 3.24 phiDir}), which does not have a direct analogue in \cite{phifem}, is term $VIII$. To bound it, we apply Cauchy-Schwarz inequality together with the interpolation estimates on $\phi  - \phi_h$ and $\chi  - \chi_h$, recalling the hypotheses $\phi \in C^{k + 1} (\Omega_h)$ and $\chi \in H^{k + 1}(\Omega_h)$:
\[ | VIII | \leqslant Ch^{k + 1} (\| w_h \|_{0, \Gamma} + | U_h | + | \phi_h   |) (\| D (v) \|_{0, \Gamma} + \| q \| _{0, \Gamma}) \,.\]
Then, to bound $w_h$ in $L^2 (\Gamma)$, we start by a trace inverse inequality $\| w_h \|_{0, \Gamma} \leqslant \frac{C}{\sqrt{h}} \| w_h \|_{0, \Omega_h}$ and apply Hardy inequality of Lemma \ref{lemma:hardy} to $w_h = \frac{\phi w_h}{\phi}$:
\[ \| w_h \|_{0, \Omega_h} \leqslant C \| \phi w_h \|_{1, \Omega_h} \leqslant   C (\| (\phi - \phi_h) w_h \|_{1, \Omega_h} + \| \phi_h w_h \|_{1,   \Omega_h}) \,. \]
Noting that, by interpolation and inverse inequalities,\[ \| (\phi - \phi_h) w_h \|_{1, \Omega_h} \leqslant C (\| \phi - \phi_h   \|_{L^{\infty} (\Omega_h)} | w_h |_{1, \Omega_h} + \| \nabla \phi - \nabla   \phi_h \|_{L^{\infty} (\Omega_h)} \| w_h \|_{0, \Omega_h}) \leqslant Ch^k   \| w_h \|_{0, \Omega_h} \]we conclude\[ (1 - Ch^k) \| w_h \|_{0, \Omega_h} \leqslant C \| \phi_h w_h \|_{1,   \Omega_h} \]
Hence, for $h$ small enough,
\[ | VIII | \leqslant Ch^{k + 1} \left( \frac{1}{\sqrt{h}} \| \phi_h w_h   \|_{1, \Omega_h} + | U_h | + | \phi_h | \right) (\| D (v) \|_{0, \Gamma} +   \| q \| _{0, \Gamma}) \,.\]
Recalling that $u_h = \phi_h w_h + \chi_h (U_h + \phi_h \times r)$, we conclude by Lemma \ref{LemDir:prop2} that $\| \phi_h w_h   \|_{1, \Omega_h}$, $| U_h |$, $| \phi_h |$ can be all bounded by  $\|u_h\|_{1,\Omega_h}$. Applying the trace inequalities to $v$ and $q$, we arrive at
\[ | VIII | \leqslant Ch^{k + 1 / 2} \| u_h \|_{1, \Omega_h} (\| v \|_{2,   \Omega} + \| q \| _{1, \Omega}) \,.\]
Since we know that $\| u_h \|_{1, \Omega_h}$ is bounded by the norms of $u$ and $p$ thanks to the already proven error estimates for the velocity in $H^1$ norm, we conclude that term $VIII$ contributes to (\ref{L2errpresque}) in the same manner as all the other terms. }

Combining (\ref{L2errpresque}) with (\ref{regvq}) proves (\ref{L2err}).

\section{Numerical tests}\label{NumSec}

In this section,  we present numerical results, first in the particular case of a fixed particle, i.e. for the Stokes equations aone in a fixed domain (cf. Appendix A and the $\phi$-FEM scheme \eqref{sch2}), and second in the case of the particulate flows (equations \eqref{eq:1a}-\eqref{intp0} and the $\phi$-FEM scheme \eqref{sch}).
These schemes will be compared with standard FEM on fitted triangular meshes \MD{as on Fig.~\ref{fig:geo mesh} left (we do not introduce higher order approximations of the curvilinear boundary of the domain, as would be the case in the isoparametric FEM for example)}.
In the case of Stokes equations, the error is measured with respect to a  manufactured solution, while a reference solution obtained by standard FEM on a fitted fine mesh is used in the case of particulate flows.
We have implemented $\phi$-FEM in \texttt{multiphenics} \cite{multiphenics}. The implementation scripts can be consulted on GitHub.\footnote{\url{https://github.com/michelduprez/phi-FEM-particulate-flows-Stokes.git}\\ or \url{https://doi.org/10.5281/zenodo.6817135}}

The fluid/solid domain in both our test cases is $\mathcal{O}=(0,1)^2\subset\mathbb{R}^2$ and we take the particle $\mathcal{S}$ as a disk  of radius $R=0.21$ centered at a point $(0.5,0.5)$. Then $\Omega=[0,1]^2\backslash \mathcal{S}$. The geometry is presented in Fig.~\ref{fig:omega_h} (left). In $\phi$-FEM, we use the following level-set function, well defined and smooth for all $(x,y)\in\mathbb{R}^2$,
\begin{equation}\label{eq:num phi}
\phi(x,y)=R^2-(x-0.5)^2-(y-0.5)^2.
\end{equation}

\MD{We present only the results with the lowest order Taylor-Hood elements, i.e. setting $k=2$ and thus using $P_2$ elements for $w_h$, $\phi_h$, $\chi_h$ (the approximation $\phi_h$ for $\phi$ is exact  in this case). The tests with elements of higher order would lead to essentially the same observations. The stabilization parameters are set to $\sigma=\sigma_u=20$ (as in \cite{phifem}).}

\subsection{Particular case of a fixed particle: Stokes equations}

We start by Stokes equations (\ref{eq:stokes}) in the domain $\Omega$, as above, with the right-hand side such that the  exact solution is as follows, cf.~\cite{alexei},
\begin{align*}
    u(x,y)&=(\cos(\pi x)\sin(\pi y), -\sin(\pi x) \cos(\pi y)) , \\
    p(x,y)&=(y-0.5)\cos(2\pi x)+(x-0.5)\sin(2\pi y),
\end{align*}
taking $\nu=1$.

\begin{figure}[ht]
\begin{center}
\includegraphics[width=0.3\linewidth]{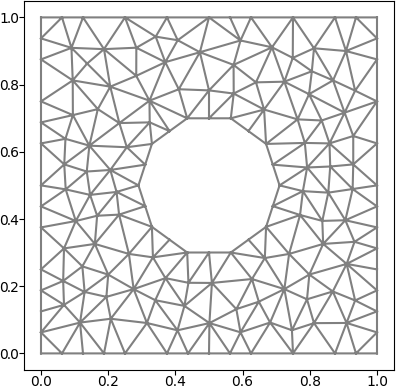}
\includegraphics[width=0.3\linewidth]{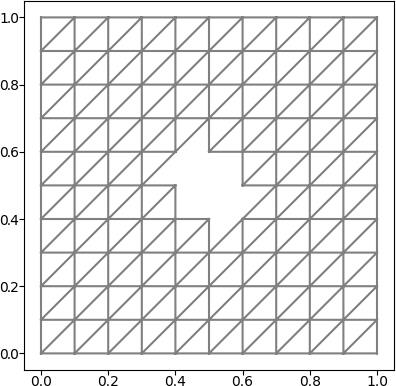}
\caption{
Mesh used for the standard FEM formulation (left) and mesh used in the $\phi$-FEM schemes (right). }
\label{fig:geo mesh}
\end{center}\end{figure}

We shall test the $\phi$-FEM scheme given by \eqref{sch2} in Appendix A and compare it with a standard Taylor-Hood FEM on a fitted mesh. To this end, we introduce a quasi-uniform triangular mesh $\mathcal{T}_h^{fit}$ fitted to $\Omega$, \AL{in the sense that the boundary nodes of the mesh lie on $\Gamma\cup\Gamma_w$.} The domain occupied by this mesh, denoted by $\Omega_h^{fit}$, is a polygonal approximation of $\Omega$, since the interface $\Gamma$ is curvilinear and cannot be represented exactly by the straight edges.
We 
introduce then the FE spaces
\begin{equation}\label{eq:def Vh} \mathcal V_h^{fit,u_D} = \left\{ \nobracket v_h \in C (\bar{\Omega}_h^{fit})^d : v_h |_T \in
   \mathbb{P}^2 (T)^d \hspace{1em} \forall T \in \mathcal{T}_h^{fit}, \hspace{1em}
   v_h = I_hu_D \tmop{ on } \Gamma_w \cap \Gamma_h^{fit} \right\}\,
\end{equation}   
 \begin{equation}\label{eq:def Mh} \mathcal M_h^{fit} = \left\{ \nobracket q_h \in C (\bar{\Omega}_h^{fit}) : q_h |_T \in
   \mathbb{P}^{1} (T) \hspace{1em} \forall T \in \mathcal{T}_h^{fit},
   \hspace{1em} \int_{\Omega_h^{fit}} q_h = 0 \right\},
   \end{equation}
where $\Gamma_h^{fit}$ is the part of the boundary of $\Omega_h^{fit}$ approximating $\Gamma$.   
A standard fitted Taylor-Hood FEM can be written as: find $(u_h,p_h)\in \mathcal V_h^{fit,u_D}\times \mathcal M_h^{fit}$ such that
\begin{equation}\label{eq:cla2}
    \int_{\Omega_h^{fit}} 2 D (u_h) : D (v_h) - \int_{\Omega_h^{fit}} p_h \Div v_h
  - \int_{\Omega_{h}^{fit}} q_h \Div u_h
  = \int_{\Omega} f v_h,
\end{equation}
  for all $(v_h,q_h)\in \mathcal V_h^{fit,0}\times \mathcal M_h^{fit}$.

Examples of meshes for the standard FEM formulation and the $\phi$-FEM scheme are given in Fig.~\ref{fig:geo mesh} (left) and (right), respectively.
In Fig.~\ref{fig:ratefix}, we report convergence results for the standard fitted Taylor-Hood  FEM \eqref{eq:cla2} and for $\phi$-FEM \eqref{sch2} in the case of Stokes equations. 
We recover the theoretical rates of convergence of  $\phi$-FEM stated in Theorem \ref{th2}: \AL{2nd order in $h$ for both the $H^1$-error in velocity and the $L^2$-error in pressure. The $L^2$-error in velocity is actually better than theoretically predicted: it is of order 3 instead of theoretically expected 2.5. We observe thus that $\phi$-FEM is fully optimal in practice: it demonstrates the same convergence rates in all the aforementioned norms as the standard FEM would demonstrate in the ideal situation of a fitted mesh on a convex polygonal domain. Actually, our setting is not ideal, $\Omega$ is neither convex, not polygonal. It is thus not surprising that the standard Taylor-Hood FEM underperforms (we recall that the mesh $\Th^{fit}$ is triangular with straight edges and no higher order geometrical approximation is introduced there). Experimentally observed convergence rates for this scheme are $\approx 2$ for the $L^2$-error in velocity, $\approx 1.5$ for the $H^1$-error in velocity, and slightly smaller than 2 for the $L^2$-error in pressure (the error in pressure is thus the only one for which the optimal convergence order seems to be retained in practice). Moreover, all the errors on all the considered meshes are systematically smaller for $\phi$-FEM than for the fitted FEM.}

\begin{rmk}\label{RmkPres}
	\MD{As already noted in Remark \ref{rmq1}, in $\phi$-FEM, it is impossible to impose $\int_\Omega p_h=0$. In our implementation, we rather impose $\int_{\Omega_h}p_h=0$ with the help of a Lagrange multiplier, i.e. we add 
	$\lambda_h\int_{\Omega_h} p_h+\mu_h\int_{\Omega_h}q_h$ (with $\lambda_h,\mu_h\in\mathbb{R}$) to the formulation. To compute the relative $L^2$-error for the pressure, we should compare $p_h$ with the exact pressure whose integral vanishes over $\Omega_h$ (recall that the pressure is physically defined up to an additive constant any way). We thus introduce  $\tilde p = p-c_{\Omega_h}$ with $c_{\Omega_h}=\frac{1}{|\Omega_h|}\int_{\Omega_h}p$ and compute the errors with respect to $\tilde p$. Similarly, in the case of standard fitted FEM, we impose $\int_{\Omega_h^{fit}}p_h=0$ by a Lagrange multiplier and compute the relative error against  $\tilde p = p-c_{\Omega_h^{fit}}$ with $c_{\Omega_h^{fit}}=\frac{1}{|\Omega_h^{fit}|}\int_{\Omega_h^{fit}}p$.}
\end{rmk}

\begin{figure}[ht]
\begin{center}
\begin{tikzpicture}[thick,scale=0.5, every node/.style={scale=1}]
\begin{loglogaxis}[xlabel=$h$
,legend pos=north west
,legend style={at={(0,1)}, font=\large, anchor=south west}
, legend columns=1]
  \addplot[color=black,mark=*,dotted,very thick] coordinates {
(0.126843364231,0.00261684125331)
(0.0636034122632,0.000597412245782)
(0.0318181380538,0.000158610228904)
(0.0159086220606,3.89047870474e-05)
(0.00795487566496,9.75411986578e-06)
 };
 \addplot[color=black,mark=triangle] coordinates {
(0.141421356237,0.000596594874732)
(0.0707106781187,7.02395633409e-05)
(0.0353553390593,6.68655697636e-06)
(0.0176776695297,5.81315318962e-07)
(0.00883883476483,6.44205915108e-08)
 };
\logLogSlopeTriangle{0.33}{0.2}{0.45}{2}{black};
\logLogSlopeTriangle{0.73}{0.2}{0.35}{3}{black};
 \legend{ $\|u-u_h\|_{0,\Omega_h^{fit}}/\|u\|_{0,\Omega_h^{fit}}$ st.-FEM, $\|u-u_h\|_{0,\Omega_h}/\|u\|_{0,\Omega_h}$ $\phi$-FEM}
\end{loglogaxis}
\end{tikzpicture}
\begin{tikzpicture}[thick,scale=0.5, every node/.style={scale=1}]
\begin{loglogaxis}[xlabel=$h$
,legend pos=north west
,legend style={at={(0,1)}, font=\large, anchor=south west}
, legend columns=1]
  \addplot[color=black,mark=*,dotted,very thick] coordinates {
(0.126843364231,0.0124156031879)
(0.0636034122632,0.00348883976582)
(0.0318181380538,0.00115184495002)
(0.0159086220606,0.000364688043354)
(0.00795487566496,0.000122570599207)
 };
 \addplot[color=black,mark=triangle] coordinates {
(0.141421356237,0.00448467460404)
(0.0707106781187,0.000881056379507)
(0.0353553390593,0.000195398382464)
(0.0176776695297,4.41540724912e-05)
(0.00883883476483,1.05000731009e-05)
 };
\logLogSlopeTriangle{0.73}{0.2}{0.35}{2}{black};
 \legend{ $\|u-u_h\|_{1,\Omega_h^{fit}}/\|u\|_{1,\Omega_h^{fit}}$ st.-FEM, $\|u-u_h\|_{1,\Omega_h}/\|u\|_{1,\Omega_h}$ $\phi$-FEM}
\end{loglogaxis}
\end{tikzpicture}
\begin{tikzpicture}[thick,scale=0.5, every node/.style={scale=1}]
\begin{loglogaxis}[xlabel=$h$
,legend pos=north west
,legend style={at={(0,1)}, font=\large, anchor=south west}
, legend columns=1]
  \addplot[color=black,mark=*,dotted,very thick] coordinates {
(0.126843364231,0.156809478242)
(0.0636034122632,0.0366934580862)
(0.0318181380538,0.0122960749514)
(0.0159086220606,0.00293071437109)
(0.00795487566496,0.000862640354336)
 };
 \addplot[color=black,mark=triangle] coordinates {
(0.141421356237,0.0736447200521)
(0.0707106781187,0.0124469934719)
(0.0353553390593,0.00267823057985)
(0.0176776695297,0.000460590935005)
(0.00883883476483,0.000102142682051)
 };
\logLogSlopeTriangle{0.73}{0.2}{0.35}{2}{black};
 \legend{ $\|\tilde p-p_h\|_{0,\Omega_h^{fit}}/\|^{fit}p\|_{0,\Omega_h^{fit}}$ st.-FEM, $\|\tilde p-p_h\|_{0,\Omega_h}/\|\tilde p\|_{0,\Omega_h}$ $\phi$-FEM}
\end{loglogaxis}
\end{tikzpicture}
\caption{Rates of convergence for the standard Taylor-Hood  FEM scheme \eqref{eq:cla2} and the $\phi$-FEM scheme \eqref{sch2} in the case of Stokes equations. The $L^2$ relative error of the velocity (left), the $H^1$ relative error of the velocity (middle) and the $L^2$ relative error of the pressure (right). }
\label{fig:ratefix}
\end{center}\end{figure}
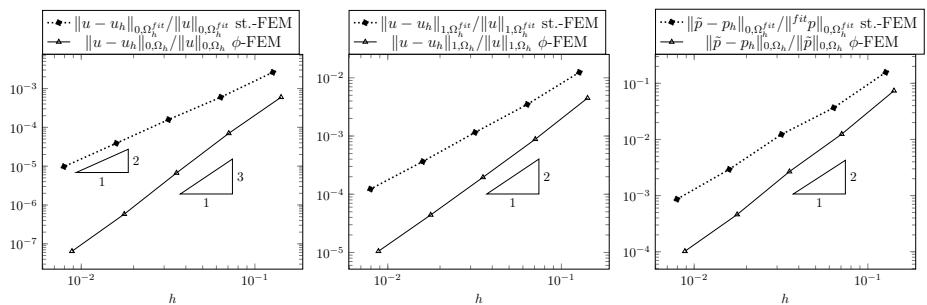

\subsection{Particulate flows}

We now turn to the creeping particulate flow equations  \eqref{eq:1a}-\eqref{intp0} in the same geometry as before. The level-set $\phi$ is again defined by \eqref{eq:num phi}. The vertical gravity is assumed to be equal to $10$. The density of the fluid and the solid are equal to $\rho_f=1$ and $\rho_s=2$, respectively, and the viscosity $\nu=1$. We deduce that the mass of the solid is equal to $m=\rho_s\pi^2R^2$.
For the cut-off $\chi$, we consider the radial polynomial of degree $5$ on the interval $(r_0,r_1)$ with $r_0=R$ and $r_1=0.45$ such that $\chi(r_0)=1$ and $\chi'(r_0)=\chi''(r_0)=\chi(r_1)=\chi'(r_1)=\chi''(r_1)=0$ \MD{so that, setting $\chi=1$ for $r<R$ and $\chi=0$ for $r>R$, the resulting $\chi$ is two times differentiable and thus $\chi\in H^{3}(\mathcal{O})$ as required by Assumption \ref{asm:chi}:} 
$$\chi(r)=\begin{cases}
1, &\text{ for } r<r_0 \\
1+\frac{f(r_0,r_1)}{(r_1-r_0)^5}, &\text{ for } r_0<r<r_1 \\
0, &\text{ for } r>r_1 \\
\end{cases}$$
where
\begin{multline*}
    f(r_0,r_1)
    =(-6r^5+15(r_0+r_1)r^4-10(r_0^2+4r_0r_1+r_1^2)r^3+30r_0r_1(r_0+r_1)r^2\\
    -30r_0^2r_1^2r + r_0^3(r_0^2-5r_1r_0+10r_1^2)).
\end{multline*}

Let us introduce a Taylor-Hood scheme which will be compared with our $\phi$-FEM scheme. We introduce first the fitting mesh $\mathcal{T}_h^{fit}$ on domain $\Omega_h^{fit}$ as in the preceding section, cf. Fig.~\ref{fig:geo mesh} (left), and adapt the Taylor-Hood FE space velocity space, cf. (\ref{eq:def Vh}), as
$$ \mathcal V_h^{fit} = \left\{ \nobracket v_h \in C (\bar{\Omega}_h^{fit})^d : v_h |_T \in
   \mathbb{P}^2 (T)^d \hspace{1em} \forall T \in \mathcal{T}_h^{fit}, \hspace{1em}
   v_h = 0 \tmop{ on } \Gamma_w \right\},$$
while keeping the pressure space (\ref{eq:def Mh}) as before.  Note that the velocity FE space does no longer contain any restrictions on the boundary part $\Gamma_h^{fit}$ approximating $\Gamma$. We shall impose the boundary conditions there with the help of Lagrange multipliers, introducing the space
$$  
  \Lambda_h^{fit}=\left\{\mu_h\in C (\bar{\Omega}_h^{fit}) :\mu_h |_F \in \mathbb{P}^{ 2} (F) \hspace{1em} \forall F \in \mathcal{F}_h^{fit}\right\},
  $$
  where $\mathcal{F}_h^{fit}$ is the set of the boundary facets on $\Gamma_h^{fit}$.
A fitted Taylor-Hood FE formulation is then written as: find $(u_h,p_h,\lambda_{h},U_h,\psi_h)\in \mathcal V_h^{fit}\times \mathcal M_h^{fit}\times \Lambda_h^{fit}\times \mathbb{R}^d\times \mathbb{R}^{d'}$ such that
\begin{multline}\label{eq:cla}
    \int_{\Omega_h^{fit}} 2 D (u_h) : D (v_h) - \int_{\Omega_h^{fit}} p_h \Div v_h
  - \int_{\Omega_{h}^{fit}} q_h \Div u_h \\
  +\int_{\Gamma_{h}^{fit}}\lambda_{h}\cdot(v_h-V_h-\omega_h\times r)
    +\int_{\Gamma_{h}^{fit}}\mu_{h}\cdot(u_h-U_h-\psi_h\times r)
   = \MD{\int_{\Omega_{h}^{fit}} \rho_fg v_h}+mg\cdot V_h,
\end{multline}
  for all $(v_h,q_h,\mu_{h},V_h,\omega_h)\in \mathcal V_h^{fit}\times \mathcal M_h^{fit}\times \Lambda_h^{fit}\times \mathbb{R}^d\times \mathbb{R}^{d'}$.
We present in Fig.~\ref{fig:sol} the velocity obtained with the standard Taylor-Hood  FEM scheme \eqref{eq:cla}. Such a velocity and the accompanying pressure, computed on a very fine fitted grid, will be used as the reference solution in the subsequent numerical experiments and will be denoted as $u,p$ in what follows.

\begin{figure}[ht]
\begin{center}
\includegraphics[width=0.3\linewidth]{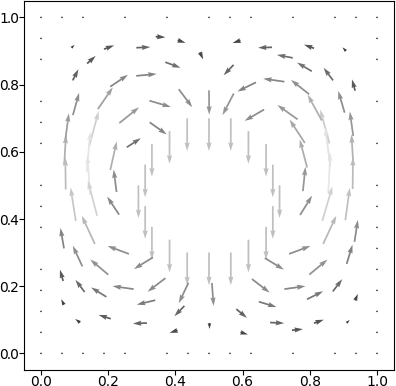}
\caption{Velocity obtained with the standard Taylor-Hood  FEM scheme \eqref{eq:cla}.}
\label{fig:sol}
\end{center}
\end{figure}

A comparison between the standard Taylor-Hood  FEM \eqref{eq:cla} and $\phi$-FEM \eqref{sch} is presented in Figs.~\ref{fig:ratemov} and \ref{fig:rate2}
\MD{(we do not report the error in the particle rotation velocity whose exact value is 0 and which is accurately predicted by all the schemes up to machine precision; this can be attributed to the symmetry of our test case)}. 
Since the error is computed with respect to a fine reference solution solution living on a fine fitted mesh, the numerical solution computed by either \eqref{eq:cla} or  \eqref{sch} should be projected to this fine mesh in order to compute the errors. This is reflected in the legends of the convergence curves: $\Omega^{fine}$ stands for the fine fitted approximation there. \AL{Similarly to Remark \ref{RmkPres}, we impose the pressure mean by a Lagrange multiplier in both schemes, and adjust the additive constants properly when computing the errors in pressure.}  

The conclusions are essentially the same as in the previous test case (Stokes equations alone): $\phi$-FEM exhibits optimal convergence rates, while the fitted standard FEM is suboptimal (with the exception of the $L^2$ error in pressure). It seems again that our theoretical estimates for the $L^2$-error of the  fluid velocity is not sharp: the experimental convergence rate is $k+1$ rather than $k+\frac 12$. The same observation can be made about the particle velocity: the experimental convergence rate is $k+1$ rather than theoretically predicted $k$.

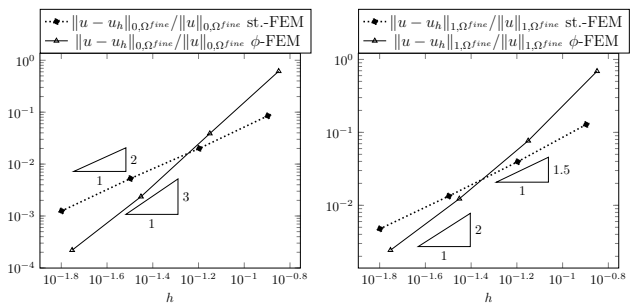
\begin{figure}[ht]
\begin{center}
\begin{tikzpicture}[thick,scale=0.5, every node/.style={scale=1}]
\begin{loglogaxis}[xlabel=$h$
,legend pos=north west
,legend style={at={(0,1)}, font=\large, anchor=south west}
, legend columns=1]
  \addplot[color=black,mark=*,dotted,very thick] coordinates {
(0.126843364231,0.0851833515147)
(0.0636034122632,0.0200412315494)
(0.0318181380538,0.00527242025755)
(0.0159086220606,0.00125160277941)
 };
 \addplot[color=black,mark=triangle] coordinates {
(0.141421356237,0.61493857136)
(0.0707106781187,0.0387925463973)
(0.0353553390593,0.00235980964656)
(0.0176776695297,0.000219727439273)
 };
\logLogSlopeTriangle{0.33}{0.2}{0.45}{2}{black};
\logLogSlopeTriangle{0.53}{0.2}{0.25}{3}{black};
 \legend{ $\|u-u_h\|_{0,\Omega^{fine}}/\|u\|_{0,\Omega^{fine}}$ st.-FEM, $\|u-u_h\|_{0,\Omega^{fine}}/\|u\|_{0,\Omega^{fine}}$ $\phi$-FEM}
\end{loglogaxis}
\end{tikzpicture}
\begin{tikzpicture}[thick,scale=0.5, every node/.style={scale=1}]
\begin{loglogaxis}[xlabel=$h$
,legend pos=north west
,legend style={at={(0,1)}, font=\large, anchor=south west}
, legend columns=1]
  \addplot[color=black,mark=*,dotted,very thick] coordinates {
(0.126843364231,0.128162791737)
(0.0636034122632,0.0395972300619)
(0.0318181380538,0.01339591578)
(0.0159086220606,0.0047667315541)
 };
 \addplot[color=black,mark=triangle] coordinates {
(0.141421356237,0.688875146104)
(0.0707106781187,0.0767538958043)
(0.0353553390593,0.0123093246107)
(0.0176776695297,0.00243009248065)
 };
\logLogSlopeTriangle{0.43}{0.2}{0.1}{2}{black};
\logLogSlopeTriangle{0.73}{0.2}{0.4}{1.5}{black};
 \legend{ $\|u-u_h\|_{1,\Omega^{fine}}/\|u\|_{1,\Omega^{fine}}$ st.-FEM, $\|u-u_h\|_{1,\Omega^{fine}}/\|u\|_{1,\Omega^{fine}}$ $\phi$-FEM}
\end{loglogaxis}
\end{tikzpicture}
\caption{Rates of convergence for the standard Taylor-Hood  FEM scheme \eqref{eq:cla} and the $\phi$-FEM scheme \eqref{sch} in the case of particulate flows. The $L^2$ relative error of the velocity (left) and the $H^1$ relative error of the velocity (right).}
\label{fig:ratemov}
\end{center}\end{figure}

\begin{figure}[ht]
\begin{center}
\begin{tikzpicture}[thick,scale=0.5, every node/.style={scale=1}]
\begin{loglogaxis}[xlabel=$h$
,legend pos=north west
,legend style={at={(0,1)}, font=\large, anchor=south west}
, legend columns=1]
  \addplot[color=black,mark=*,dotted,very thick] coordinates {
(0.126843364231,0.00765002232217)
(0.0636034122632,0.00178176724726)
(0.0318181380538,0.000473835012931)
(0.0159086220606,0.000118493755294)
 };
 \addplot[color=black,mark=triangle] coordinates {
(0.141421356237,0.0799691585045)
(0.0707106781187,0.00607843739558)
(0.0353553390593,0.00104879698873)
(0.0176776695297,0.000180057562223)
 };
\logLogSlopeTriangle{0.73}{0.2}{0.25}{2}{black};
 \legend{ $\|p-p_h\|_{0,\Omega^{fine}}/\|p\|_{0,\Omega^{fine}}$ st.-FEM, $\|p-p_h\|_{0,\Omega^{fine}}/\|p\|_{0,\Omega^{fine}}$ $\phi$-FEM}
\end{loglogaxis}
\end{tikzpicture}
\begin{tikzpicture}[thick,scale=0.5, every node/.style={scale=1}]
\begin{loglogaxis}[xlabel=$h$
,legend pos=north west
,legend style={at={(0,1)}, font=\large, anchor=south west}
, legend columns=1]
  \addplot[color=black,mark=*,dotted,very thick] coordinates {
(0.126843364231,0.107916009683)
(0.0636034122632,0.0255882918524)
(0.0318181380538,0.00683654255282)
(0.0159086220606,0.00162490193875)
 };
 \addplot[color=black,mark=triangle] coordinates {
(0.141421356237,0.365547593002)
(0.0707106781187,0.0243072040108)
(0.0353553390593,0.00142857620397)
(0.0176776695297,0.000136875666604)
 };
\logLogSlopeTriangle{0.33}{0.2}{0.35}{2}{black};
\logLogSlopeTriangle{0.53}{0.2}{0.2}{3}{black};
 \legend{ $|U-U_h|/|U|$ st.-FEM, $|U-U_h|/|U|$ $\phi$-FEM}
\end{loglogaxis}
\end{tikzpicture}
\caption{Rates of convergence for the standard Taylor-Hood  FEM scheme \eqref{eq:cla2} and the $\phi$-FEM scheme \eqref{sch2} in the case of particulate flows. The $L^2$ relative error of the pressure (left) and  relative error of the displacement of the solid 
(right).}
\label{fig:rate2}
\end{center}\end{figure}
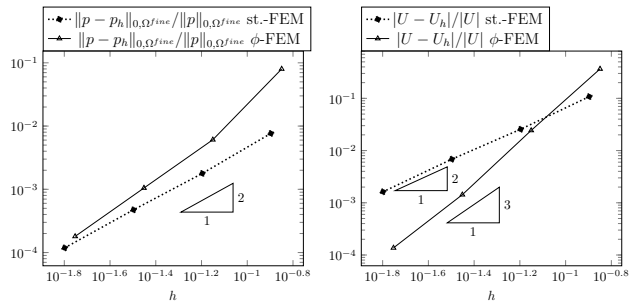

\section{Conclusions and perspectives} 
The main goal of the $\phi$-FEM approach is  to avoid the non standard quadrature on the cut mesh cells inherent to CutFEM. Some of the attractive features of $\phi$-FEM are:
\begin{itemize}
	\item $\phi$-FEM is readily available for finite elements of any order (without the need of any additional higher order approximation of the geometry). 
	\item (corollary of the previous point) $\phi$-FEM outperforms the standard fitted FEM on polygonal meshes if the order of piecewise polynomials is $>1$ (unless one implements more complicated versions of FEM in order to treat properly the curvilinear boundary, such as the isoparametric FEM).
	\item  $\phi$-FEM uses standard FE spaces and is based on a variational formulation of the problem, so that it can be easily implemented in existing general-purpose FEM libraries, provided they allow one to compute the jumps on selected facets and the second derivatives on selected cells.
\end{itemize}	
To counter-balance the last point, we should mention that the variational formulation at the base of $\phi$-FEM can be rather complicated. Typically, it contains more terms than a fitted FEM or a CutFEM scheme for the same problem. The implementation, although straightforward, may thus require some extra programming efforts. As a reward, one easily gets a good description of the geometry which may result in smaller computing times than those for standard FEM, as highlighted in \cite{cotin2021phi}.  

In the present article we have adapted $\phi$-FEM to the Stokes equations and to the combination of the Stokes equations with the motion of a rigid particle inside the fluid, providing a first brick in future applications of this technology to fluid structure interaction problems.

Of course, there remains a lot of open questions. To conclude, we list some of those  (not necessarily related to the particular case of Stokes equations or flows with particles) and envisage potential solutions:
\begin{itemize}
    \item Theoretical suboptimality of $\phi$-FEM in the $L^2$-norm. However, all the numerical experiments show the optimal convergence in this norm, which gives us hope that a sharper theoretical estimate could be found.
    \item A mismatch between the theoretical construction of the active mesh $\Th$ and its practical implementation, cf. Remark \ref{PractGeom}. The theoretical construction of $\Th$ is crucial for the current proof of coercivity, cf. Lemma \ref{lemma:coer}.  An alternative proof should be found. 
    \item A practical construction of the levelset function, which should satisfy some assumptions summarized in Remark \ref{RmkAsmPhi}. In the present article, $\phi$ was given analytically, but in more realistic applications one will have to construct an appropriate $\phi_h$ on the discrete level directly. A good candidate, in the vicinity of $\Gamma$ would be the signed distance to $\Gamma$, for which there exist efficient libraries, but it will remain to extend it in a smooth way (again directly on the discrete level) far from $\Gamma$. As an alternative, we note that in other versions of $\phi$-FEM, $\phi_h$ may be required only on mesh cells near $\Gamma$, cf. \cite{phifem2, cotin2021phi}. 
    \item The current construction of stabilization in $\phi$-FEM relies heavily on the linearity of the governing equations. Indeed, the terms with prefactor $\sigma$ in (\ref{sch}) reproduce the governing equations for both trial and test functions. If the equations are non-linear, one cannot do this since the formulation should remain linear in the test functions. Thus, going from Stokes to Navier-Stokes, for example, is not straightforward. Various options of linearization of the stabilization terms should be yet tested numerically and theoretically. 
\end{itemize}

\appendix


\section{$\phi$-FEM for Stokes equations in a fixed domain}

In this section, we propose a $\phi$-FEM scheme for the simpler case of a fixed solid in the fluid.
The governing equations are the non-homogeneous Stokes equations given by
\begin{equation}\label{eq:stokes}
\left\{\begin{array}{ll}
- 2\Div(D( u)) + \nabla p = f, &\tmop{in } \Omega, \\
 \Div u = 0,& \tmop{in }  \Omega, \\
 u = u_D, &\tmop{on } \Gamma\cup\Gamma_w.
\end{array}\right.
\end{equation}

Assume that $u_D$ and $f$ are defined in the whole discrete domain $\Omega_h$. Inspired by the $\phi$-FEM scheme for particulate flow given in \eqref{sch}, we can derive the following $\phi$-FEM scheme for the non-homogeneous Stokes equations \eqref{eq:stokes}:
find $w_h \in \mathcal V_h$, $p_h \in \mathcal M_h$ satisfying
\begin{equation}
  \label{sch2}
 \int_{\Omega_h} 2 D (u_D + \phi_h w_h) : D
  (\phi_h s_h) - \int_{G_h} (2 D (u_D + \phi_h w_h)
  - p_hI) n \cdot \phi_h s_h \text{}
\end{equation}
\[ - \int_{\Omega_h} p_h \Div (\phi_h s_h) - \int_{\Omega_h} q_h \Div (u_D  + \phi_h w_h)+\sigma_u J_u(u_D  +\phi_h w_h,\phi_h s_h) \]
\[   + \sigma h^2  \sum_{T \in
   \mathcal{T}_h^{\Gamma}} \int_T (- \Delta (u_D  +\phi_h w_h) + \nabla p_h) \cdot (-
   \Delta (\phi_h s_h) - \nabla q_h) \]
\[ + \sigma \sum_{T \in \mathcal{T}_h^{\Gamma}} \int_T \Div (u_D  +\phi_h w_h) \Div (\phi_h v_h)
  \]
\[ = \int_{\Omega_h} f \phi_h s_h   + \sigma h^2  \sum_{T \in \mathcal{T}_h^{\Gamma}} \int_T f (-
   \Delta (u_D+\phi_h s_h) -\nabla q_h), \]
for all $s_h \in \mathcal V_h$, $q_h \in \mathcal M_h$.

We now state our second main result for the Stokes equations:
 \begin{theorem}\label{th2}
Suppose that Assumptions \ref{asm0}, \ref{asm1}, \ref{asm2} and \ref{asmStokes} hold true, the mesh $\mathcal{T}_h$ is quasi-uniform.
Let $(u,p)\in H^{k+1}(\Omega)^d\times H^{k}(\Omega)$  be the solution to \eqref{eq:stokes} and $(w_h,p_h)\in \mathcal V_h\times \mathcal M_h$ be the solution to \eqref{sch2}.  Denoting $u_h:=\phi_h w_h $, it holds
\begin{equation*}
  | u - u_h|_{1, \Omega\cap\Omega_h}+\AL{\frac{1}{\nu}}| p - p_h|_{0, \Omega\cap\Omega_h}  \le Ch^k (\|u\|_{k+1,\Omega}+ \AL{\frac{1}{\nu}}\|p\|_{k,\Omega})
\end{equation*}
with a constant $C>0$ depending on the $C_0$, $m$, $M$ in Assumptions \ref{asm0}, \ref{asm2}, on the maximum of the derivatives of $\phi$, on the mesh regularity, and on the polynomial degree $k$, but independent of $h$, $f$, and $u$. \\
Moreover, supposing $\Omega\subset\Omega_h$
\begin{equation*}
 \MD{  \| u - u_h\|_{0, \Omega} \le Ch^{k+1/2}
 (\|u\|_{k+1,\Omega}+ \frac{1}{\nu}\|p\|_{k,\Omega})}
 \end{equation*}
with a constant $C>0$ of the same type.
\end{theorem}
The proof of Theorem \ref{th2} can be adapted from the proof of Theorem
\ref{th1}. It is even more simple.

\section{\AL{A glossary of geometrical notations.}}\label{AppGlossary}
\AL{\begin{tabular}{l|l}
     $\Gamma_h$ & the approximate interface: $\Gamma_h=\{\phi_h=0\}$ \\
     $\Th$ & the active mesh: $\Th=\{T\in\Th^\mathcal{O}:T\cap\{\phi_h<0\}\neq\varnothing\}$ \\
     $\Th^\Gamma$ & intersection of  $\Th$ with $\Gamma_h$  $\Th^{\Gamma}=\{T\in\Th:T\cap\Gamma_h\neq\varnothing\}$ \\
     $\Th^{\Gamma,ext}$ & $\mathcal{T}_h^{\Gamma}$ and the cells which are neighbors  and neighbors of neighbors of cells \\&\hfill of $\mathcal{T}_h^{\Gamma}$ in  $\Th$\\
     $\Th^i$ & $\mathcal{T}_h \setminus  \mathcal{T}_h^{\Gamma}$ \\
     $\mathcal{F}_h^\Gamma$ & $ \mathcal{F}_h^{\Gamma} = \{E \text{ (an internal facet of } \mathcal{T}_h)
  \text{ such that } \exists T \in \mathcal{T}_h^\Gamma \text{ and } E \in \partial T\}.$\\
     $G_h$ & the internal component of $\partial\Omega_h$, corresponding to the interface $\Gamma$: $G_h=\partial\Omega_h\setminus\Gamma_w$ \\
     $B_h$ & the strip between $\Gamma$ and $G_h$: $B_h = \Omega_h \setminus \Omega$\\
     $B_h^{\Gamma}$ & the strip between $\Gamma$ and $\Gamma_h$ \\
     $\tilde{B}_h$ & the strip between $\Gamma_h=\{\phi_h=0\}$ and $G_h$: $\tilde{B}_h = \{\phi_h > 0\} \cap \Omega_h$ \\
\end{tabular}}

\AL{We also recall that the domain occupied by the active mesh $\Th$ is denoted by  $\Omega_h$, i.e. $\Omega_h:=\left(\cup_{T\in\Th}T\right)^o$. The same convention is applied to the submeshes $\Th^\Gamma$ and $\Th^i$, giving respectively $\Omega_h^\Gamma$ and $\Omega_h^i$.
}

%

\end{document}